\documentclass[10pt]{amsart}
\usepackage{graphicx} % Required for inserting images
\usepackage[utf8]{inputenc}
\usepackage{amsmath}
\usepackage{amssymb}
\usepackage{amsthm}
\usepackage{graphicx}
\usepackage[rightcaption]{sidecap}
\usepackage{MnSymbol} 
\usepackage{color}
\usepackage[square,numbers]{natbib}
\usepackage{nicefrac}
\usepackage{mathdots}
\usepackage{multirow}
\usepackage{nicefrac}
\usepackage{faktor}
\usepackage{hyperref}

\DeclareUnicodeCharacter{2212}{-}

\newtheorem{theorem}{Theorem}[section]
\newtheorem{theorem*}[theorem]{*Theorem}
\newtheorem{example}[theorem]{Example}
\newtheorem{corollary}[theorem]{Corollary}
\newtheorem{lemma}[theorem]{Lemma}
\newtheorem{proposition}[theorem]{Proposition}
\newtheorem{remark}[theorem]{Remark}

\newtheorem{proposition*}[theorem]{*Proposition}
\newtheorem{lemma*}[theorem]{*Lemma}
\newtheorem{corollary*}[theorem]{*Corollary}
\newtheorem{example*}[theorem]{Example}

\theoremstyle{definition}
\newtheorem{definition}[theorem]{Definition}

\title{On abelian complex structures on Nilpotent Lie Algebras}

\author{L. Pizarro}
\address[L. Pizarro and R. Villacampa]{Departamento de Matem\'aticas\,-\,I.U.M.A.\\
Universidad de Zaragoza\\
Campus Plaza San Francisco\\
50009 Zaragoza, Spain}
\email{lpizarro@unizar.es}
\email{raquelvg@unizar.es}

\author{R. Villacampa}

\begin{document}

\maketitle

\tableofcontents

\bigskip

\begin{abstract}
We provide a new approach to the classification, up to equivalence, of abelian complex structures on nilpotent Lie algebras.  As an application, we recover the already known classifications in dimensions 4 and 6 and provide a classification in dimension 8 for 1-abelian complex structures.  Some results about the general case are given. 

%ANTIGUO - The classification up to equivalence of abelian nilpotent complex structures of 6-dimensional Nilpotent Lie Algebras (NLA) were studied in ([Hermitian], [Invariant]). The objective of this article is arriving to the last mentioned results defining a $GL(2, \mathbb{C})$-action over $M_{2\times 2}(\mathbb{C})$ and looking for the orbit set of that action.
\end{abstract}

\section{Introduction}
A complex structure on a real Lie algebra $\mathfrak{g}$ is an endomorphism
$J \colon \mathfrak{g} \to \mathfrak{g}$ satisfying $J^2 = -\mathrm{Id}$ together
with the integrability condition
\[
N_J(x,y) := [Jx,Jy] - [x,y] - J[Jx,y] - J[x,Jy] = 0, \qquad x,y \in \mathfrak{g}.
\]
When $\mathfrak{g}$ is the Lie algebra of a simply connected Lie group $G$, such a $J$ induces a left-invariant complex structure on $G$, and when $G$ is furthermore a nilpotent Lie group admitting a lattice $\Gamma$, it induces an invariant complex structure on the compact nilmanifold $\Gamma \backslash G$. Nilmanifolds equipped with invariant complex structures constitute one of the main sources of examples and counterexamples in complex non-K\"ahler geometry, and their study links Lie theory, complex geometry, geometric analysis and mathematical physics.

Among  complex structures on nilpotent Lie algebras, a distinguished
class is given by the \emph{abelian} complex structures, namely those satisfying
the stronger condition
$$
[Jx,Jy] = [x,y], \qquad x,y \in \mathfrak{g},
$$
which in particular implies integrability. Equivalently, a complex structure $J$ is abelian if the
eigenspace decomposition $\mathfrak{g}_{\mathbb{C}} = \mathfrak{g}_{1,0} \oplus
\mathfrak{g}_{0,1}$ satisfies $[\mathfrak{g}_{1,0},\mathfrak{g}_{1,0}] = 0$, so
that $\mathfrak{g}_{1,0}$ is an abelian subalgebra of $\mathfrak{g}_{\mathbb{C}}$.  

The abelian condition is far from being a mild refinement of integrability: it
imposes genuine structural restrictions on the underlying nilpotent Lie
algebra. On the geometric side, Rollenske \cite{Rollenske2009} showed that the
ascending central series of any nilpotent Lie algebra carrying an abelian
complex structure is $J$-invariant, so that the associated nilmanifold
automatically inherits the structure of an iterated holomorphic principal
torus bundle; this is not the case for a general (non-abelian) invariant
complex structure. On the algebraic side, Barberis and Dotti
\cite{BarberisDotti2004} obtained an explicit obstruction to the existence of abelian complex structures on 2-step nilpotent Lie algebras showing that the abelian
condition constrains the relative size of the derived algebra with respect to
the center.

%: a two-step
%nilpotent Lie algebra $\mathfrak{n}$ with center $\mathfrak{z}$ and
%$n = \dim \mathfrak{n} - \dim \mathfrak{z} \geq 3$ cannot carry an abelian
%complex structure if the bracket, regarded as a skew-symmetric map on
%$\mathfrak{n}/\mathfrak{z}$, is generic in the sense that
%$2\dim[\mathfrak{n},\mathfrak{n}] = n(n-1)$; this shows that the abelian
%condition constrains the relative size of the derived algebra with respect to
%the center. 
In addition, Barberis, Dotti and Verbitsky \cite{BDV2009} proved
that nilmanifolds with an abelian complex structure have holomorphically
trivial canonical bundle, a property with strong consequences for the
associated hypercomplex and HKT geometry. On the other hand, the abelian
condition does not by itself bound the nilpotency step: Barberis and Dotti
\cite{BarberisDotti2004} showed that for every $k$ there exists a $k$-step
nilpotent Lie algebra carrying an abelian \emph{hypercomplex} structure, so
that obstructions of the type above arise from the finer interaction between
the bracket and $J$ at each step, rather than from the step itself.

Abelian complex structures were singled out because of their especially rigid
interaction with the Lie bracket, and they appear naturally in several contexts:
they are the building blocks of abelian hypercomplex structures
\cite{BarberisDotti1996}, they underlie explicit constructions
of hyper-K\"ahler with torsion metrics \cite{FinoGrantcharov}, and the Dolbeault
cohomology of a nilmanifold carrying an abelian complex structure can be computed
purely at the level of the Lie algebra \cite{ConsoleFinoPoon}, \cite{CFGU2000}. Deformations of
abelian complex structures on nilmanifolds were studied by Maclaughlin-Pedersen-Poon-Salamon
and by Rollenske \cite{MPPS2006,Rollenske2009} further underlining the special role this
class plays among invariant complex structures.

The study of invariant complex structures on nilpotent Lie algebras was put
on a systematic footing in 2001, independently and essentially simultaneously,
by Salamon \cite{Salamon} and by Cordero, Fern\'andez, Gray and Ugarte
\cite{CFGU2001}. Salamon \cite{Salamon} gave an algebraic characterization,
in terms of a suitably adapted basis, of which nilpotent Lie algebras admit a
complex structure, while Cordero, Fern\'andez, Gray and Ugarte
\cite{CFGU2001} studied general (not necessarily abelian) nilpotent complex
structures with a view towards computing Dolbeault cohomology at the level of
the Lie algebra. Building on this foundation, Cordero, Fern\'andez and Ugarte
\cite{CFU2002} obtained the first classification of six-dimensional nilpotent
Lie algebras admitting an \emph{abelian} complex structure, together with a
description of the invariant Hermitian and symplectic geometry compatible
with such structures. Shortly afterwards, Barberis and Dotti
\cite{BarberisDotti2004} extended the study of abelian complex structures
beyond the nilpotent setting, to a class of solvable Lie algebras, and related
them to the associated hypercomplex geometry. These two lines of work were
unified and refined, from the point of view of moduli, by Andrada, Barberis
and Dotti \cite{ABD2011}, who classified all six-dimensional Lie algebras (nilpotent as well as non-nilpotent solvable) admitting an abelian complex
structure, and parameterized, on each of these algebras, the space of such
structures up to holomorphic equivalence (\cite{ABD}). More recently, a complementary
approach, based on the Fr\"olicher spectral sequence and on the classification
of arbitrary (not necessarily abelian) invariant complex structures on
six-dimensional nilpotent Lie algebras, was carried out by Ceballos, Otal,
Ugarte and Villacampa \cite{COUV}, who located the abelian ones within the
full moduli space of invariant complex structures and analyzed the associated
Hermitian metrics.

The six-dimensional case has been particularly influential because it represents the first dimension where the classification becomes genuinely nontrivial. In higher dimensions the situation becomes considerably more intricate. The diversity of nilpotent Lie algebras increases dramatically, and complete classifications are no longer available. Very recently, Andrada and Vera \cite{AndradaVera} introduced an
explicit procedure to construct 2-step nilpotent Lie algebras equipped with
complex structures starting from finite graphs, characterizing those graphs
admitting an \emph{abelian} adapted complex structure.

In this paper we focus on nilpotent Lie algebras $\frak g$ of dimension $2n$ endowed with an abelian complex structure characterized by the condition $h^{1,0} = n-1$, where $h^{1,0}= \dim H_{\bar \partial}^{1,0}(\frak g)$.  We call these structures \emph{1-abelian complex structures} and they are in one-to-one correspondence with complex matrices.  More precisely, the generic complex expression of a  1-abelian complex structure is given by the following equations, which we denote by $J_M$:
$$
    J_M = \begin{cases}
            d\omega^1=...=d\omega^{n-1}=0,\\
            d\omega^n=\omega M \overline{\omega}^T,
        \end{cases}
    $$
    where $\{\omega^i\}_{i=1}^n$ is a basis of $(1,0)$-forms,  $\omega := \begin{pmatrix} 
        \omega^1 & \omega^2 & ... & \omega^{n-1}
    \end{pmatrix}$ and $M \in M_{n-1}(\mathbb{C})$.

We study equivalence of complex structures $J_M$ using first the classification of $M$ by *congruence, following the methods developed by \cite{HS1} and \cite{HS2}, depending on the regularity of the matrix $M$.  With this new approach, we recover the previously known classifications in dimensions 4 and 6, provide the complete classification in dimension 8 and give a lower bound for the non-equivalent families of 1-abelian complex structures that exist in each complex dimension $n$.

The paper is organized as follows. In Section~\ref{sec:2} we
recall the basic notions on complex structures on Lie algebras, with
particular emphasis on the abelian case.  We define 1-abelian complex structures and classify them when $M$ is Hermitian.  In
Section~\ref{sec:3} we revisit the Horn-Sergeichuk theory of *congruence and present the canonical forms of *congruence for regular and non-regular matrices.  Finally, in Section~\ref{sec:4} we obtain the complete classification of 1-abelian complex structures in dimensions 4, 6 and 8 (see Table~\ref{Tabla-dim8} for the summary of dimension 8).  Moreover, we provide a combinatorial approach to determine a lower bound for the number of non-equivalent families of 1-abelian complex structures for each complex dimension $n$ (see Theorem~\ref{thm:cota-complex-st}).

%%%%%%%%%%%%%%%%%%%%%%%%%%%%%%%%%%%%%%%%

\section{Abelian complex structure equations on Nilpotent Lie Algebras}\label{sec:2}
Let $\mathfrak{g}$ be a Lie algebra, $\mathfrak{g}_\mathbb{C} = \mathfrak{g}\otimes_{\mathbb R}\mathbb C$ its complexification and $\mathfrak{g}^{*}_\mathbb{C}$, the dual of the latter. Given a complex structure $J$ on $\mathfrak{g}$, we denote by $\mathfrak{g}_{1,0}$ and $\mathfrak{g}_{0,1}$ the eigenspaces of eigenvalues $i$ and $-i$ respectively.  Extending $J$ to $\mathfrak{g}^{*}_\mathbb{C}$, there is a natural bigraduation $\bigwedge^{*}\mathfrak{g}^{*}_\mathbb{C} = \bigwedge^p \mathfrak{g}^{1,0} \bigoplus \bigwedge^q \mathfrak{g}^{0,1}=\bigoplus_{p,q} \bigwedge^{p,q} \mathfrak{g}^{*}$. If $d:\bigwedge^{*}\mathfrak{g}^{*}_\mathbb{C} \longrightarrow \bigwedge^{*+1} \mathfrak{g}^{*}_\mathbb{C}$ is the extension of the Chevalley-Eilenberg differential operator, then $J$ satisfies that $d(\mathfrak{g}^{1,0}) \subseteq \mathfrak{g}^{2,0}\oplus\mathfrak{g}^{1,1}$.\\
In the particular case when $\mathfrak{g}$ is a \textit{nilpotent} Lie algebra (NLA for short) of dimension $2n$ for $n\geq1$, Salamon proved in \cite{Salamon} that the integrability condition for $J$ is equivalent to the existence of a basis $\{\omega^i\}_{i=1}^n$ for $\mathfrak{g}^{1,0}$ such that $d\omega^1=0$ and $$d\omega^i \in \mathcal{I}(\omega^1,...,\omega^{i-1})\text{ for } i=2,...,n.$$ 
Moreover, if the complex structure is of \textit{nilpotent} type, then \cite{CFGU2000} there exists a basis $\{\omega^i\}_{i=1}^n$ for $\mathfrak{g}^{1,0}$ such that $d\omega^1=0$ and
\begin{equation*}%\label{nilp}
    d\omega^i \in \bigwedge^2\langle\omega^1,...,\omega^{i-1},\omega^{\bar 1},...,\omega^{\overline{i-1}}\rangle\text{ for } i=2,...,n.
\end{equation*} 
An important class of nilpotent complex structures is that of the \textit{abelian complex structures} consisting of those satisfying $[JX, JY]=[X, Y]$, and this implies that the subalgebra $(\mathfrak{g}_{1,0}; [,])$ is abelian.  In terms of the differential of $(1,0)$-forms, it happens that $d(\mathfrak{g}^{1,0}) \in \mathfrak{g}^{1,1}$ and therefore, there exists a basis $\{\omega^i\}_{i=1}^n$ for $\mathfrak{g}^{1,0}$ such that $d\omega^1=0$ and
\begin{equation}\label{abel}
    d\omega^i \in \langle\omega^1,...,\omega^{i-1}\rangle \wedge \langle\omega^{\bar 1},...,\omega^{\overline{i-1}}\rangle \text{ for } i=2,...,n.
\end{equation}

We will focus on a special type of abelian complex structures. 

\begin{definition}
Let $(\mathfrak{g}, J)$ be a $2n$-dimensional NLA endowed with an abelian complex structure. 
    We will say that $J$ is a \emph{1-abelian complex structure} if $h^{1,0} = n-1$, where $h^{1,0}= \dim H_{\bar \partial}^{1,0}(\frak g)$.
\end{definition}

\begin{example}
    Using the classification of abelian complex structures in six-dimensional NLAs \cite{COUV} we observe that the only Lie algebras admitting 1-abelian complex structures are $\mathfrak h_i$ for $i=2, 3, 4, 5$ and $8$. 
\end{example}

The following lemma provides a characterization of 1-abelian complex structures:

\begin{lemma}
Let $(\mathfrak{g}, J)$ be a $2n$-dimensional NLA endowed with a 1-abelian complex structure.  Then, there exists a basis $\{\omega^i\}_{i=1}^n$ for $\mathfrak{g}^{1,0}$ such that the complex structure equations are of the form 
\begin{equation}\label{ouracs}
    \begin{cases}
            d\omega^1=...=d\omega^{n-1}=0\\
            d\omega^n=\omega M \overline{\omega}^T
        \end{cases}
    \end{equation}
    being $\omega := \begin{pmatrix} 
        \omega^1 & \omega^2 & ... & \omega^{n-1}
    \end{pmatrix}$ and $M \in M_{n-1}(\mathbb{C})$ not being the zero-matrix.
\end{lemma}

\begin{proof}
Since $J$ is abelian, then $d(\mathfrak{g}^{1,0}) \in \mathfrak{g}^{1,1}$ and therefore, for any $(1,0)$-form $\alpha$, $d\alpha = \bar\partial \alpha$.  If $h^{1,0}=n-1$, then $\dim d(\mathfrak{g}^{1,0}) = 1$ as vector space and we obtain the result using \eqref{abel}.
\end{proof}

Observe that a 1-abelian complex structure can be identified with the matrix $M$.  We will use the notation $J_M$ referring to equations~\eqref{ouracs}.

%There exists interesting examples of NLA admitting 1-abelian complex structures. 
%\begin{example}
%    Consider the 6-dimensional Lie algebra $\mathfrak{g}$ given by the non-zero Lie brackets
%    $$[X_1, X_2]=[X_3, X_4]=-X_6,$$ where $\mathfrak{g}=\langle X_1,\ldots, X_6\rangle$ as real vector space. This NLA is known as \textit{Heisenberg algebra}.
%
 %   In terms of the dual basis $\{e^j\}_{j=1}^6$ of $\mathfrak{g}^*$ and using the exterior differential of forms, the structure equations of $\mathfrak g$ are:
%$$de^1=\dots=de^5=0; \quad de^6=e^1\wedge e^2+e^{3}\wedge e^4.$$ 
%Let us consider the complex structure given by $$Je^1=-e^2; \quad Je^3=-e^4; \quad Je^5=2e^6.$$Then $\mathfrak{g}^{1,0}=\langle \omega^1:=e^1+ie^2, \, \omega^2:=e^3+ie^4, \, \omega^3:=e^5-2ie^6\rangle$. And the equations of the complex structure are $$
%        d\omega^1=d\omega^2=0,\quad
%        d\omega^3=\omega^{1\bar 1}+\omega^{2\bar 2}.$$
%This structure follows equations~\eqref{ouracs} with $M=I_2$.       
%\end{example}

\medskip

The existence of a 1-abelian complex structure imposes topological obstructions to the nilpotent Lie algebra $\mathfrak g$. In particular: %An interesting property of this type of abelian complex structures is the following. 

\begin{proposition}\label{prop:topo}
    Let $(\mathfrak{g}, J_M)$ be a 1-abelian complex structure on a 2n-dimensional NLA, then:
    \begin{itemize}
    \item[(a)] $\mathfrak{g}$ is 2-step nilpotent. 
    \item[(b)] $b_1(\mathfrak g)\geq 2n-2$.
    \end{itemize}
\end{proposition}

\begin{proof} Let $\{\omega^i\}_{i=1}^n$ be the adapted basis for the structure $J_M$. For case (a), denote by $\{Z_i\}_{i=1}^n$ the dual basis of $\{\omega^i\}_{i=1}^n$.
    If $\mathfrak{g}$ has a complex structure of type \eqref{ouracs} then the complexification of the center has at least two elements, $Z_n$ and $\bar Z_n$, and thus $\mathfrak{g}_2^{\mathbb{C}}=\mathfrak{g}$. So $\mathfrak{g}$ is $s$-step nilpotent with $s \leq 2$.  For case (b), $\{\mathfrak{Re}\, \omega^i, \mathfrak{Im} \,\omega^i\}_{i=1}^{n-1}$ are real closed 1-forms in $\frak g$.
\end{proof}
However, the converse statements are false as the next example shows:

\begin{example}
Consider the 8-dimensional Lie algebra $\mathfrak{g}$ given by the non-zero Lie brackets
    $$[X_1, X_2]=-X_7,\quad [X_3, X_4]=-X_8,$$ where $\mathfrak{g}=\langle X_1,\ldots, X_8\rangle$ as a real vector space.  It is straightforward to check that $\mathfrak{g}_1=Z(\mathfrak{g})=\langle X_5, \, X_6, \, X_7, \, X_8 \rangle$ and $\mathfrak{g}_2=\mathfrak{g}$, so $\mathfrak{g}$ is 2-step. 

\medskip

    In terms of the dual basis $\{e^j\}_{j=1}^8$ of $\mathfrak{g}^*$ and using the exterior differential of forms, the structure equations of $\mathfrak g$ are:
$$de^1=\dots=de^6=0; \quad de^7=e^{12},\quad de^8 =e^{34},$$ 
where $e^{ij}$ stands for $e^i\wedge e^j$. Now, consider the complex structure defined as $$Je^1=-e^2; \quad Je^3=-e^4; \quad Je^5=2e^7; \quad Je^6=2e^8.$$ Then $\mathfrak{g}^{1,0}=\langle \omega^1,\dots,\omega^4\rangle$, where $$\omega^1=e^1+ie^2; \quad \omega^2=e^3+ie^4; \quad \omega^3=e^5-2ie^7; \quad \omega^4=e^6-2ie^8.$$
    The equations of this complex structure are \begin{equation*}
           d\omega^1=d\omega^2=0,\quad
           d\omega^3=\omega^{1\bar 1},\quad
           d\omega^4=\omega^{2\bar 2},
    \end{equation*}
   and it is immediate to see that $h^{1,0} = 2\neq 3$, so this structure is not 1-abelian (although it is abelian).  
\end{example}

%\begin{remark}
%    From now on, we use the following notation, for instance $\mathfrak{g}=(0^5, 12)$ means that there exists a basis $\{x^i\}_{i=1}^6$ of $\mathfrak{g}^{*}$ such that $dx^1=dx^2=dx^3=dx^4=dx^5=0$ and $dx^6=x^1 \wedge x^2$. Or, in terms of the Lie bracket and the dual basis $\{x_i\}_{i=1}^6$ of $\mathfrak{g}$  , $[x^1, x^2]=-x^6$ and the rest of the brackets are zero.
%\end{remark}

Moreover, we can also establish some partial results about the behaviour of the Frölicher spectral sequence associated to NLAs endowed with 1-abelian complex structures.

\begin{proposition}\label{prop:espectral}
    Let $(\mathfrak{g}, J_M)$ be a $1$-abelian complex structure on a $2n$-dimensional NLA.  If $b_1(\mathfrak g) = 2n-2$, then $E_1(\mathfrak g) \neq E_{\infty}(\mathfrak g)$.
\end{proposition}

\begin{proof}
    For general abelian complex structures, we have that $h^{0,1} = n$.  Now, for the special case of 1-abelian, we obtain that
    $$E_1^{|1|} = E_1^{1,0} + E_1^{0,1} = h^{1,0} + h^{0,1} = 2n-1.$$  Since $E_1^{|1|}\geq E_{\infty}^{|1|} = b_1$, the result follows.
\end{proof}

%Returning to our question, two complex structures $J_1$ and $J_2$ in a given nilpotent Lie algebra $\mathfrak{g}$ are equivalent if and only if there exists a $\mathbb{C}$-linear isomorphism $F:\mathfrak{g}^{1,0}_{J_1} \longrightarrow \mathfrak{g}^{1,0}_{J_2}$ that commutes with the Chevalley-Eilenberg differential operator. With this in mind, given a 6-dimensional nilpotent Lie algebra $\mathfrak{g}$ endowed with an abelian complex structure $J$, then $J$ is fully determined by the values of $d\omega^i$, for $\{\omega^i\}_{i=1}^3$ a given basis of $\mathfrak{g}^{1,0}$. So another complex structure $J'$ is equivalent if there exists a basis change such that the condition for the Chevalley-Eilenberg differential is still satisfied. The definition of $d\omega^i$ is given by the condition $d(\mathfrak{g}^{1,0}) \subseteq \mathfrak{g}^{1,1}$.\\

As for other types of complex structures, we are concerned with studying two types of problems: which NLAs admit 1-abelian complex structures and the classification of the structures up to equivalence.  Recall that two complex structures $J_1$ and $J_2$ are equivalent if there exists a $\mathbb{C}$-linear isomorphism $F:\mathfrak{g}^{1,0}_{J_1} \longrightarrow \mathfrak{g}^{1,0}_{J_2}$ that commutes with the Chevalley-Eilenberg differential operator.  In dimension 6, the classification up to isomorphism of abelian complex structures that a given nilpotent Lie algebra can admit was made in \cite{ABD, COUV}.  Some partial results in higher dimensions appear in \cite{AndradaVillacampa2016} where Hermitian balanced metrics are involved.

\medskip

For 1-abelian complex structures, the conditions for having an equivalence between $J_M$ and $J_{\widetilde M}$ are equivalent to the 
existence of two bases of (1,0)-forms $\{\omega^i\}_{i=1}^n$ and $\{\sigma^i\}_{i=1}^n$ such that
\begin{itemize}
    \item[(i)] $\omega=\sigma P$, with $\sigma := \begin{pmatrix} 
        \sigma^1 & \sigma^2 & ... & \sigma^{n-1}
    \end{pmatrix}$ and $P \in GL(n-1, \mathbb{C})$,
    \item[(ii)] $c\,\omega^n=\sigma^n$ with $c \in \mathbb{C}^*$,
    \item[(iii)] satisfying: $\widetilde M=c(PMP^*)$.\end{itemize} 
In terms of the new basis $\{\sigma^i\}_{i=1}^n$, equations \eqref{ouracs} transform into 
\begin{equation*}\label{cseq2}
    \begin{cases}
            d\sigma^1=...=d\sigma^{n-1}=0\\
            d\sigma^n= c(\sigma PM\bar P^T\bar \sigma^T).
    \end{cases}
\end{equation*}
\begin{remark}
    From now on, given a matrix $N$, we will denote $\bar N^T$ by~$N^*$.
\end{remark}

In fact, the equivalence of 1-abelian complex structures can be set in terms of the corresponding matrices:
 
 \begin{lemma}\label{equivalence}
 Two complex structures $J_M$ and $J_{\widetilde M}$ are \emph{equivalent} if and only if there exist $c \in \mathbb{C} \setminus \{0\}$ and $P \in GL(n-1, \mathbb{C})$ such that $$\widetilde M=c(PMP^*).$$ 
 \end{lemma}
 We can think of a two-step process of transforming the matrix $M$:
\begin{equation}\label{2-step-M} 
M\stackrel{(i)}{\longrightarrow} M'= PMP^*\stackrel{(ii)}{\longrightarrow} \widetilde M = cM' = c(PMP^*).
\end{equation}
 Observe that condition (ii) can be viewed as rescaling the last element of the new basis, whereas condition (i) corresponds to a known $GL(n-1, \mathbb{C})$-action:
\begin{definition}\label{*conj}
    Let $GL(k,\mathbb{C})$ denote the group of invertible $k\times k$ complex matrices, and let $M_k(\mathbb{C})$ denote the set of all $k\times k$ complex matrices.  
The \textbf{*conjugacy group action} of $GL(k,\mathbb{C})$ on $M_k(\mathbb{C})$ is the map
\[
GL(k,\mathbb{C}) \times M_k(\mathbb{C}) \to M_k(\mathbb{C}), 
\quad (P,M) \mapsto PM P^*.
\]
\end{definition}

\begin{definition}\label{def:matrices-congruentes}
Let $A, B\in M_n(\mathbb C)$.  We will say that $A$ and $B$ are \emph{*congruent} if there exists $P\in GL(n,\mathbb C)$ such that $B = PAP^*$. 
\end{definition}

Observe that *congruent matrices provide equivalent complex structures.  However, the converse is false due to possibility of scaling the matrix by a complex number.

In what follows, we will study first the quotient space $M_n(\mathbb{C})/\sim$ of square matrices modulo *conjugacy group action, characterizing each orbit in terms of algebraic invariants of the matrix $M$ (see Section~\ref{sec:3}). Then, in Section~\ref{sec:4}, we will come back to the problem of classifying  1-abelian complex structures taking into account the rescaling condition (ii).  Before studying the general case, let us analyze the 1-abelian complex structures $J_M$ defined by a Hermitian matrix $M$.

\subsection{1-Abelian complex structures associated to Hermitian matrices}  In this section we classify up to equivalence 1-abelian complex structures $J_M$ where $M$ is a Hermitian matrix, i.e. $M\in M_k(\mathbb C)$ satisfying $M=M^*$.  In this setting, we can apply the well-known Sylvester's inertia law:
\begin{theorem}\label{sil}
    Let $M \in M_k(\mathbb{C})$ be a Hermitian complex matrix. Then, there exists $P \in GL(k, \mathbb{C})$ such that \[PMP^*=\begin{pmatrix}
        Id_p &  & \\
         & -Id_q & \\
        &  & 0_r
    \end{pmatrix}\]
    where $p+q+r=k$, where $p$, respectively $q$, is the number of positive, resp. negative,  eigenvalues of~$M$ and $r=dim(\ker (M))$. 
\end{theorem}

Observe that the pair $(p,q)$ is the signature of $M$. Moreover, these three values remain invariant under *congruence transformation. 

\medskip

Sylvester's inertia law provides a direct method to classify 1-abelian complex structures when $M$ is a Hermitian matrix.  In particular, given $J_M$ with $M$ Hermitian, we can assume that $M$ is of the canonical form: 
\begin{equation}\label{Mherm}
    M=\begin{pmatrix}
        Id_p &  & \\
         & -Id_q & \\
        &  & 0_r
    \end{pmatrix},
    \end{equation}
where now $p+q=n-1-r$. We will identify $J_M$ (or directly $M$) with the tuple $(p, q, r)$.  Concerning complex structures under equivalence, the following result holds:

\begin{lemma}
The complex structures $(p, q, r)$ and $(q, p, r)$ are equivalent.
\end{lemma}
\begin{proof}
    Let $(p, q, r)$ be the 1-abelian complex structure determined by the matrix $M$ given by~\eqref{Mherm}.
Observe that $M$ belongs to the same *congruence class as $$M'=\begin{pmatrix}
        -Id_q &  & \\
         & Id_p & \\
        &  & 0_r
    \end{pmatrix},$$
   since both matrices have the same rank and the same signature. Now, in terms of the new basis $\{\sigma^i\}_{i=1}^n$ given by $\sigma^i = \omega^i$, for $i=1,\ldots, n-1$ and rescaling the last element as $\sigma^n=-\omega^n,$ we obtain a 1-abelian complex structure with matrix $-M'$, providing an equivalence between the structures $(p,q,r)$ and $(q,p, r)$. 
\end{proof}

Next, we establish the main result for 1-abelian complex structures in the Hermitian case, where we can always assume that $p\geq q$:

\begin{theorem}\label{teo:Hermitian}
    Let $(\mathfrak{g}, J)$ be a $2n$-dimensional NLA with a 1-abelian complex structure of type \eqref{ouracs} where $M$ is a Hermitian matrix. Then there exists a basis $\{\sigma^i\}_{i=1}^n$ of $\mathfrak{g}^{1,0}$ such that the complex structure equations are one of the following 
    \begin{equation}\label{eq:Hermitian}
    J_{n,k,r} = \begin{cases}
            d\sigma^1=...=d\sigma^{n-1}=0,\\
            d\sigma^n= \displaystyle{\sum_{j=1}^k \sigma^{j\bar j}}- \displaystyle{\sum_{j=k+1}^{n-1-r} \sigma^{j\bar j}},
    \end{cases}
\end{equation}
with $k=\lfloor \frac{n-1-r}{2} \rfloor, \dots, n-1-r$ if $n-r$ is odd, or  $k=\lfloor \frac{n-1-r}{2} \rfloor+1, \dots, n-1-r$ if $n-r$ is even, and where empty sums are understood to vanish (in particular when $k+1>n-1-r$).
\end{theorem}

Finally, let us determine the real Lie algebras underlying the complex structures $J_{n,k,r}$:
%\subsection{Underlying Lie algebras}
\begin{theorem}
    The real Lie algebras underlying equations \eqref{eq:Hermitian} are:
    \begin{equation}\label{gk}
\mathfrak{h}_{n,r} = \begin{cases}
de^1=\cdots = de^{2n-1}=0,\\
de^{2n} = \displaystyle{\sum_{k=1}^{n-1-r}e^{2k-1} \wedge e^{2k}}.
\end{cases}
\end{equation}
\end{theorem}

\begin{proof}
For convenience, let us express $d\sigma^n$ in \eqref{eq:Hermitian} as follows:
$$d\sigma^n= \displaystyle{\sum_{j=1}^{n-1-r} \varepsilon_j\sigma^{j\bar{j}}},$$
where $\varepsilon_j = \pm 1$.
To obtain equations \eqref{gk}, just consider the following real basis:
\begin{eqnarray*}\sigma^j &=& \varepsilon_j\,e^{2j-1} + i\,e^{2j},\quad j=1,\ldots, n-1,\\
\sigma^{n}& = &e^{2n-1} - i\,\frac{e^{2n}}{2}.
\end{eqnarray*}
\end{proof}

It is worth remarking that there exists only one Lie algebra for each $r$.  Moreover, the number of non-equivalent complex structures is $\lfloor \frac{n+1-r}{2} \rfloor$.

\begin{remark}
    The Lie algebras $\mathfrak h_{n,r}$ given by \eqref{gk} are related to the well-known Heisenberg Lie algebras.  Recall that the \textit{n-dimensional Heisenberg algebra} (\cite{Y} Definition 1.1), denoted by $\mathfrak{h}_n$ is a $(2n+1)$-dimensional Lie algebra for which there exists a basis $\{e_1,e_2,\dots,e_{2n},e_{2n+1}\}$ with bracket relations given by $$[e_i, e_{2n+1}]=[e_{2n+1},e_{2n+1}]=0;$$ $$[e_i,e_j]=\begin{cases}
    0 \quad \text{ if }j \neq n+i.\\
    e_{2n+1} \quad \text{ if }j=n+i
\end{cases}.$$
%Considering the dual basis and the differential operator, then $\mathfrak{h}_n$ can be expressed in terms of the differentials of the elements of the dual basis as follows:
%$$de^{2n+1}=\sum_{i=1}^n e^i \wedge e^{i+n},$$ $$de^j=0 \quad \text{ for all }1\leq j \leq 2n.$$
The relation between Lie algebras $\mathfrak{h}_{n,r}$ given by~\eqref{gk} and the Heisenberg algebras is the following:  $$\mathfrak{h}_{n,0}=\mathfrak{h}_{n-1} \oplus \mathfrak{a}_1$$ and, in the general case
\begin{equation}\label{eq:relation-heis}
\mathfrak{h}_{n,r}=\mathfrak{h}_{n-r,0}\oplus \mathfrak{a}_{2r}=\mathfrak{h}_{n-r-1} \oplus \mathfrak{a}_{2r+1},
\end{equation} where $\mathfrak{a}_{k}$ is the abelian Lie algebra of dimension $k$.

%In~\cite{Sant}, the betti numbers of $\mathfrak{h}_n$ are calculated.  In particular, $$\dim H^{m}(\mathfrak{h}_n) = \begin{pmatrix}
%2n\\ m
%\end{pmatrix} - \begin{pmatrix}
%2n\\ m-2
%\end{pmatrix}, \quad m\leq n.$$  Applying the Künneth formula for de Rham cohomology, it is possible to compute the betti numbers for the Lie algebras $\mathfrak{h}_{n,r}$.
\end{remark}
%
%\begin{proposition}
%Let us denote by $b_i^n = \dim H^{m}(\mathfrak{h}_n)$ and $\beta_i^{n,r} = \dim H^{m}(\mathfrak{h}_{n,r})$.  Then:
%$$\beta_i^{n,r} = \sum_{p+q=i} \left(\begin{pmatrix}
%2(n-r-1)\\ p
%\end{pmatrix} - \begin{pmatrix}
%2(n-r-1)\\ p-2
%\end{pmatrix}\right) \begin{pmatrix}
%2r+1\\q 
%\end{pmatrix},$$ whenever $p\leq n-r-1$ and $q\leq 2r+1$.
%\end{proposition}

Finally, there also exist non-Hermitian matrices producing complex structures that are equivalent to the ones defined by Hermitian matrices.  In particular:

\begin{corollary}\label{rmk:multiple-Hermitian}
If a 1-abelian complex structure $J_M$ is given by a matrix $M$ that is a multiple of a Hermitian matrix $N$, i.e. $M = c\,N$, then $J_M\sim J_N\sim J_{n,k,r}$ for the tuple $(n,k,r)$ associated to $N$.
\end{corollary}

\begin{proof}
Observe that $M = c\,N$ implies that $J_M\sim J_N$, rescaling the last element of the basis adapted to $J_M$.
\end{proof}

%%%%%%%%%%%%%%%%%%%%%%%%%%%%%%%%%%%%%%%%%%%%%%%%%%%%%%%%%

\section{Non-Hermitian matrices: Revisiting the Horn–Sergeichuk theory of *congruence}\label{sec:3}
In this section we deal with the classification of matrices up to *congruence (see Definition~\ref{def:matrices-congruentes}). 
In their influential works, Horn and Sergeichuk, \cite{HS1, HS2}, developed a systematic 
description of this relation and determined the canonical forms that 
complex square matrices can admit under *congruence. 
Their results provide a complete classification framework which serves as a powerful tool in the study of equivalence problems for matrices over the complex field. %{\color{red}Our objective in this section is to review the theorems by Horn and Sergeichuk about classification of square matrices module *congruence and try to adapt and apply then into our classification problem.}
Our objective in this section is to review the theorems and constructions by Horn and Sergeichuk about classification of square matrices modulo *congruence for the regular and the singular case, completing some missing details.

Their main theorem is the following: 
\begin{theorem}\label{HS}\cite[Theorem 1 (b)]{HS1}: Each square complex matrix $A$ is *congruent to a matrix $B$, that is a direct sum, uniquely determined up to permutation of summands, of canonical matrices of the following three types: \begin{itemize}
    \item \textbf{Type 0}: $J_n(0)$;
    \item \textbf{Type I}: $\lambda\Delta_n$ with $|\lambda|=1$;
    \item \textbf{Type II}: $H_{2n}(\mu)$ with $|\mu|>1$;
\end{itemize}
where 

\[
J_n(\lambda) =
\begin{bmatrix}
\lambda & 1 &  &   & 0 \\
 & \lambda & 1 &  &   \\
 &   & \ddots & \ddots &   \\
 &   &   & \lambda & 1 \\
0 &   &   &   & \lambda
\end{bmatrix}
\quad (J_1(\lambda) = [\lambda]);
\]

\[
\Delta_n =
\begin{bmatrix}
0 &        &        & 1\\
  &  &   \iddots     & i \\
  & 1       & \iddots &  \\
1 & i &  & 0
\end{bmatrix}.
\quad (\Delta_1=[1]);
\]

\[
H_{2n}(\mu) =
\begin{bmatrix}
0 &  I_n   \\
J_n(\mu)  &  0

\end{bmatrix}.
\quad (H_2(\mu)=\begin{bmatrix}
0 &  1   \\
\mu  &  0

\end{bmatrix}).
\]
\end{theorem}

\medskip

 In particular, every regular matrix is *congruent to a direct sum of Type I and Type II blocks and this decomposition is unique up to permutation of summands. The blocks of Type 0 control the singular part of the matrix.

\medskip

In the following subsections we will review separately the algorithms developed by  Horn and Sergeichuk to determine the *congruence canonical form of a given matrix $A$, depending on whether it is regular or not.   We will explain how to construct the matrix $B$ in the previous theorem obtaining in this way our matrices $M'$ in \eqref{2-step-M}.

\subsection{Canonical form of *congruence for regular matrices}\label{section3.1}

Here, we review and complete some missing details in \cite{HS1}. We start by introducing a concept that will be crucial in the development of the theory:

\begin{definition}\label{*cosquare}
    Let $A \in GL(n, \mathbb{C})$ be a regular matrix.   Its *\textit{cosquare} $C(A)$ is defined as 
    \begin{equation}\label{eq:cosquare}C(A):=(A^*)^{-1}A.\end{equation}
\end{definition}

%\begin{remark}\label{hermrec}
    From this definition, a regular matrix $A$ is Hermitian if and only if $C(A)=I_n$.
%\end{remark}
 So one can think of the *cosquare as a measure of the non-hermitianity of a matrix.  Moreover, $C(A)$ is also a regular matrix since $A$ is so. 

\medskip

The next result is a simple fact about linear algebra but will be relevant in the future: 
\begin{lemma}\label{eigenvalue}
    Let $A \in GL(n, \mathbb{C})$ be a regular matrix.   If $z$ is an eigenvalue of $C(A)$, then $\frac{1}{\overline{z}}$ is also an eigenvalue of $C(A)$.
\end{lemma}
\begin{proof}
    Let $z$ be an eigenvalue for $C(A)$ and $v$ a corresponding eigenvector. Then $C(A)v=zv$. Using \eqref{eq:cosquare} we obtain: \[(A^*)^{-1}Av=zv \Longleftrightarrow Av=zA^*v \Longleftrightarrow (A-zA^*)v=0.\]
    As this is verified for any eigenvector associated to $z$, then $det(A-zA^*)=0$. Applying that for a general complex matrix $\overline{det(M)}=det(\overline M)$, we obtain that $det((A-zA^*)^*)=det(A^*-\bar zA)=0.$  Now,
    \[det(A^*-\bar zA)=0 \Longleftrightarrow det\left(-\bar z(A-\frac{1}{\overline{z}}A^*)\right)=(-\bar z)^ndet\left(A-\frac{1}{\overline{z}}A^*\right)=0.\]
    As $z\neq 0$, since $A \in GL(n, \mathbb{C})$, we get that $det(C(A)-\frac{1}{\overline{z}}I_n)=0$, so $\frac{1}{\overline{z}}$ is an eigenvalue of $C(A)$.
\end{proof}

\begin{remark}\label{remark-eigenvalues}
Note that $0$ cannot be an eigenvalue of $C(A)$ because it is not a singular matrix. The previous lemma tells us that in the set $sp(C(A))$ there are two families of elements: on the one hand, we have pairs of eigenvalues $\left(z, \frac{1}{\overline{z}}\right)$ where $|z| \neq 1$ and, on the other hand the eigenvalues $w$ with $|w|=1$ (it is straightforward that $|w|=1 \Rightarrow w=\frac{1}{\overline{w}})$. These two sets will play a key role to determine the *congruence canonical form of a regular matrix $A$. 
\end{remark}

In the following results, the relation between the type of *congruence of $M$ and $C(M)$ is studied.  Recall that two matrices $A, B \in GL(n, \mathbb{C})$ are \textit{similar} if there exists another $M \in GL(n, \mathbb{C})$ such that $A=M^{-1}BM$. The canonical forms under this equivalence are the Jordan forms.

\begin{lemma}\label{cosquaresim}
    If $A$ is *congruent to $B$, then $C(A)$ is similar to $C(B)$.
\end{lemma}

\begin{proof}
    If $A$ is *congruent to $B$, then there exists $M \in GL(n, \mathbb{C})$ such that $A=MBM^*$ and hence $$C(A)=C(MBM^*)=((MBM^*)^*)^{-1}(MBM^*)=(M^*)^{-1}C(B)M^*$$
    So choosing $S=M^*$, we have that $C(A)=S^{-1}C(B)S$ and therefore, $C(A)$ and $C(B)$ are similar.
\end{proof}
However, the other direction is not always true, as it is shown in the next example
\begin{example}(See \cite[Section 3]{HS1}):
    Consider $M=\begin{pmatrix}
        i
    \end{pmatrix}$ and $M'=\begin{pmatrix}
        -i
    \end{pmatrix}$ both as $1\times 1$-matrices. It is straightforward to show that $C(M)=C(M')=\begin{pmatrix}
        -1
    \end{pmatrix}$ however there cannot exist a matrix $S=\begin{pmatrix}
        k
    \end{pmatrix}$, with $k\in\mathbb C^*$, such that $SMS^*=M'$. Indeed, $SMS^*=\begin{pmatrix}
        k
    \end{pmatrix}\begin{pmatrix}
        i
    \end{pmatrix}\begin{pmatrix}
        \bar k
    \end{pmatrix}=\begin{pmatrix}
        |k|^2i
    \end{pmatrix} \neq \begin{pmatrix}
        -i
    \end{pmatrix}$.  
\end{example}

In the previous example, $M\not\sim_*M'$ but $M\sim_*-M'$.  This ``-" will be of special relevance as the next result shows:

\begin{theorem}\label{lemma5}\cite[Lemma 5]{HS1}: Let $A,B\in GL(n,\mathbb C)$ be matrices with similar *cosquares, that is, $C(A)=S^{-1}C(B)S$ for some non-singular S. Let $B_S:=S^{*}BS$, let $M:=B_SA^{-1}$, and suppose that $M$ has k real negative eigenvalues, counted according to their algebraic multiplicities $(0\leq k\leq n)$. Then: \begin{enumerate}
    \item $M$ is similar to a real matrix.
    \item There are square complex matrices $D_{-}$ and $D_{+}$ of size $k$ and $n-k$, respectively, such that $A$ is *congruent to $(-D_-)\oplus D_+$ and $B$ is *congruent to $D_- \oplus D_+$. 
\end{enumerate}
\end{theorem}
Given how $M$ is defined, one can think of $M$ as a measure of how far $B$ fails to be *congruent to $A$. In fact if $M=I_n$ then $B \sim A$.
A sketch of the proof for (1) is given. For a complete proof, see \cite{HS1}.
\begin{proof}(Sketch of)
    It is straightforward to prove that $C(A)=C(B_S)$, so $$(B_S^*)^{-1}B_S=(A^*)^{-1}A \Leftrightarrow B_SA^{-1}=B_S^*(A^*)^{-1}=(A^{-1}B_S)^*$$ so $M^*=A^{-1}B_S=A^{-1}MA$ and ($1$) is proved (a property of real matrices is that if a matrix is similar to its conjugate transpose, then it is similar to a real matrix). %For a proof of $\textit{2}$, see \cite{HS1} Lemma 5.
\end{proof}

This result, and its proof, give us an algorithm in order to look for the canonical *congruent form for a maximal rank matrix. Given $A \in GL(n, \mathbb{C})$, the idea is to build another matrix $B$ such that $C(A) \sim_{sim} C(B)$, and then apply Theorem \ref{lemma5}.  We need a technical lemma concerning matrices $\Delta_n$ and $H_{2n}(\mu)$ before presenting the algorithm:

%Given $\mu \in \mathbb{C}$ with $|\mu|=1$, consider the blocks $H_{2n}$ and $\Delta_n$ then the following result is verified
\begin{lemma}\label{simil-cosquare}
Let $\mu \in \mathbb{C}$ with $|\mu|>1$.  Then:
\begin{itemize}
\item[(i)] $C(H_{2n}(\mu)) \sim_{sim} J_n(\mu)\oplus J_n(\bar \mu^{-1})$;
\item[(ii)] $H_{2m}(\mu) \sim_* -H_{2m}(\mu)$;
\item[(iii)] $C(\Delta_n)\sim_{sim}J_n(1)$;
\item[(iv)] $C(e^{i\nicefrac{\theta}{2}}\Delta_n) \sim_{sim} J_n(e^{i\theta})$.
\end{itemize}
\end{lemma}
\begin{proof}
An easy calculation gives us $$C(H_{2n}(\mu))=(H_{2n}(\mu)^*)^{-1}H_{2n}(\mu)=\begin{bmatrix}
    0 & I_n\\
    (J_n(\mu)^*)^{-1} & 0
\end{bmatrix} \begin{bmatrix}
    0 & I_n\\
    J_n(\mu) & 0
\end{bmatrix}=$$ $$=J_n(\mu)\oplus (J_n(\mu)^*)^{-1}.$$ 
Next step is to show that $(J_n(\mu)^*)^{-1} \sim_{sim}J_n(\bar \mu^{-1})$. For this purpose we calculate $(J_n(\mu)^*)^{-1}$:
$$(J_n(\mu)^*)^{-1}=\begin{bmatrix}
\bar \mu &  &  &   & 0 \\
1 & \bar \mu &  &  &   \\
 & 1  & \ddots & \ddots &   \\
 &   & \ddots  & \bar \mu &  \\
0 &   &   & 1  & \bar \mu
\end{bmatrix}^{-1}=\begin{bmatrix}
\bar \mu^{-1} &  &  &   & 0 \\
 & \bar \mu^{-1} &  &  &   \\
 &   & \ddots &  &   \\
 &   &   & \bar \mu^{-1} &  \\
* &   &   &   & \bar \mu^{-1}
\end{bmatrix}.$$
We observe that $J_n(\bar \mu^{-1})$ is a lower triangular matrix, so it has a sole eigenvalue, $\bar \mu^{-1}$ of multiplicity $n$. In addition, $rk((J_n(\mu)^*)^{-1}-\bar \mu^{-1}I_n)=n-1$, then $(J_n(\mu)^*)^{-1} \sim_{sim}J_n(\bar \mu^{-1})$, 
and (i) is proved.  

\medskip

Case (ii) is a direct computation.  Just observe: $$\begin{bmatrix}
        I_m & 0\\
        0 & -I_m
    \end{bmatrix}\begin{bmatrix}
        0 & I_m\\
        J_m(\mu) & 0
    \end{bmatrix}\begin{bmatrix}
        I_m & 0\\
        0 & -I_m
    \end{bmatrix}=-\begin{bmatrix}
        0 & I_m\\
        J_m(\mu) & 0
    \end{bmatrix}.$$

For (iii) we calculate first $C(\Delta_n)$:
$$C(\Delta_n)=\begin{bmatrix}
0 &        &       -i & 1\\
  &\iddots  &   \iddots     &  \\
-i  & 1       &  &  \\
1 &  &  & 0
\end{bmatrix}^{-1}\begin{bmatrix}
0 &        &        & 1\\
  &  &   \iddots     & i \\
  & 1       & \iddots &  \\
1 & i &  & 0
\end{bmatrix}=\begin{bmatrix}
1 & 2i &  &   & * \\
 & 1 & 2i &  &   \\
 &   & \ddots & \ddots &   \\
 &   &   & 1 & 2i \\
0 &   &   &   & 1
\end{bmatrix}.$$
Let us study the similarity of $C(\Delta_n)$. On the one hand, as it is an upper triangular matrix, then $p_z(C(\Delta_n))=(1-z)^n$, so $Sp(C(\Delta_n))=\{1\}$ with multiplicity $n$. On the other hand, $rk(C(\Delta_n)-I_n)=n-1$, so the geometric multiplicity of $z=1$ is $1$, then $C(\Delta_n) \sim_{sim} J_n(1)$.  

Finally, for case (iv), a short calculation shows that $$C(e^{i\nicefrac{\theta}{2}}\Delta_n)=e^{i\theta}C(\Delta_n).$$ Using (iii), $e^{i\theta}C(\Delta_n)\sim_{sim} e^{i\theta}J_n(1)$. It is clear that $e^{i\theta}J_n(1) \sim_{sim} J_n(e^{i\theta})$, indeed, $e^{i\theta}J_n(1)$ has $e^{i\theta}$ as unique eigenvalue of multiplicity $n$. In addition $e^{i\theta}(J_n(1)-Id)$ has rank $n-1$. This concludes the result. 
\end{proof}

Now, we present the algorithm for regular matrices:

    \begin{enumerate}
    \item[(a)] Given $A \in GL(n, \mathbb{C})$, diagonalize by similarity its *cosquare $C(A)$, i.e., determine its Jordan canonical form $J(C(A))$.  For that, suppose that $\{e^{i\phi_j}\}_{j=1}^p$ are the eigenvalues of modulus~1, and $\{\mu_i\}_{i=1}^q$ are the eigenvalues of modulus greater than~1. Note that, by Lemma \ref{eigenvalue}, $\bar \mu_i^{-1}$ is also an eigenvalue with the same multiplicity as $\mu_i$.  Now, we have that: 
    $$C(A) \sim_{sim}  \underbrace{\bigoplus_{i=1}^q (J_{m_i}(\mu_i) \oplus J_{m_i}(\bar \mu_i^{-1}))\oplus \bigoplus_{j=1}^pJ_{n_j}(e^{i\phi_j})}_{J(C(A))}.$$   
    %\item Suppose that $$C(A) \sim_{sim}  \bigoplus_{i=1}^q (J_{m_i}(\mu_i) \oplus J_{m_i}(\bar \mu_i^{-1}))\oplus \bigoplus_{j=1}^pJ_{n_j}(e^{i\phi_j})$$ where $|\mu_i|>1$, $0\leq \phi_j < 2\pi$. Here $m_i$ is the algebraic multiplicity of the eigenvalue $\mu_i$ and $n_j$ are the algebraic multiplicities of the eigenvalues of modulus 1. Note that, by lemma \ref{eigenvalue}, $\bar \mu_i^{-1}$ is also eigenvalue with the same multiplicity as $\mu_i$. Then we define $$B:= \bigoplus_{i=1}^q H_{2m_i}(\mu_i)\oplus \bigoplus_{j=1}^pe^{i\frac{\phi_j}{2}}\Delta_{n_j}$$ 
    where $J_{m_i}(\mu_i)$ refers to the Jordan blocks associated to the eigenvalue $\mu_i$ and the same for $J_{n_j}(e^{i\phi_j})$ (see Remark~\ref{rmk:Jordan} for details).
    \medskip
    \item[(b)] Define $$B:= \bigoplus_{i=1}^q H_{2m_i}(\mu_i)\oplus \bigoplus_{j=1}^pe^{i\frac{\phi_j}{2}}\Delta_{n_j}.$$  By Lemma~\ref{simil-cosquare}, we have that: $$C(B)\sim_{sim}\bigoplus_{i=1}^q (J_{m_i}(\mu_i) \oplus J_{m_i}(\bar \mu_i^{-1}))\oplus \bigoplus_{j=1}^pJ_{n_j}(e^{i\phi_j}) \sim_{sim} C(A),$$ and therefore $A$ and $B$ have similar *cosquares.
    \medskip
    \item[(c)] %Is straightforward to prove that $C(B)\sim_{sim}\bigoplus_{i=1}^q (J_{m_i}(\mu_i) \oplus J_{m_i}(\bar \mu_i^{-1}))\oplus \bigoplus_{j=1}^pJ_{n_j}(e^{i\phi_j}) \sim_{sim} C(A)$, so $A$ and $B$ have similar *cosquares, that is, there exist $S \in GL(n, \mathbb{C})$ such that $C(A)=S^{-1}C(B)S$. Define $B_S=S^{*}BS$ and $M=B_SA^{-1}$, 
    By Theorem~\ref{lemma5} there are square complex matrices $D_{-}$ and $D_{+}$ such that $A \sim (-D_-)\oplus D_+$ and $B \sim D_- \oplus D_+$. For our purposes, we may assume that $A$ will be *congruent to $$A\sim_*\bigoplus_{i=1}^q \delta_iH_{2m_i}(\mu_i)\oplus \bigoplus_{j=1}^p\varepsilon_je^{i\frac{\phi_j}{2}}\Delta_{n_j}, \quad \delta_i, \, \varepsilon_j  \in \{-1,1\},$$
    for some choice of signs $\delta_i, \, \varepsilon_j$. Moreover, by Lemma~\ref{simil-cosquare} (iii), we can assume that $\delta_i=1$ for $i=1,\dots,q$. Then $A$ is *congruent to $$A\sim_*\bigoplus_{i=1}^q H_{2m_i}(\mu_i)\oplus \bigoplus_{j=1}^p\varepsilon_je^{i\frac{\phi_j}{2}}\Delta_{n_j},$$ where    $\varepsilon_i \in \{-1, 1\}, \, |\mu_i|>1, \, 0 \leq \phi_j \leq 2\pi.$
    \end{enumerate}
    %\item (PONERLO COMO LEMA A PARTE) Note that $H_{2m}(\mu) \sim -H_{2m}$, indeed 
    
    %So, we can assume that $\delta_i=1$ for $i=1,\dots,q$. Then $A$ is *congruent to $$\bigoplus_{i=1}^q H_{2m_i}(\mu_i)\oplus \bigoplus_{j=1}^p\varepsilon_je^{i\frac{\phi_j}{2}}\Delta_{n_j}, \quad \varepsilon_i \in \{-1, 1\}, \, |\mu_i|>1, \, 0 \leq \phi_j \leq 2\pi.$$

\medskip

\begin{remark}\label{rmk:Jordan}
It is worth noticing that the geometric multiplicity of an eigenvalue $\mu_i$ ($e^{i\phi_j}$), determines precisely the total number of Jordan blocks associated with it in the Jordan canonical form of $C(A)$. While the algebraic multiplicity determines the sum of the sizes of these blocks, the geometric multiplicity counts the individual blocks themselves, regardless of their dimensions. This means that for a fixed eigenvalue $\mu_i$ or $e^{i\phi_j}$ there may exist more than one Jordan block.
\end{remark}

\begin{remark}
In order to choose the correct value of $\varepsilon_j$, we need to construct the matrix $M$ given in Theorem~\ref{lemma5} and count the number of negative eigenvalues.
However, there are some particular cases of *congruence: if $M$ has no negative eigenvalues then $D_-$ does not appear and $A$ is *congruent to $B$, whereas if all their eigenvalues are negative, then $D_+$ is zero and $A$ is *congruent to $-B$. 
\end{remark}

\begin{remark}\label{hermHS}
    The previous algorithm recovers Sylvester's inertia law (Theorem \ref{sil}) when $M$ is a regular Hermitian matrix.  Since $C(M)=I_n=J_1(1) \oplus \stackrel{n}{\dots} \oplus J_1(1)$, then $M \sim _* \varepsilon_1 \Delta_1 \oplus \stackrel{n}{\dots} \oplus \varepsilon_n \Delta_n$ where $\varepsilon_i \in \{-1, 1\}$.
%\end{remark}

%\begin{remark}\label{rmk:max-eigen-value}
Moreover, suppose that $C(A)$ has a unique eigenvalue $\mu_0$ of maximum multiplicity $n$ and that the dimension of its corresponding eigenspace is also $n$.  Then, $C(A)\sim_{sim} \mu_0 I_n = \bigoplus^n J_1(\mu_0)$. In particular, $\exists\, P\in GL(n,\mathbb C)$ such that $PC(A)P^{-1} = \mu_0 I_n$.  This implies that $C(A) = \mu_0 I_n$ and since $C(A) = (A^*)^{-1}A$, then $A = \mu_0 A^*$.  On the other hand, since $|\mu_0| = 1$ (see Remark~\ref{remark-eigenvalues}), we can express $\mu_0 = e^{i\theta}$.  Now, $\widetilde A:= e^{-i\theta/2} A$ is a Hermitian matrix and in particular, the original $A$ is a multiple of a Hermitian matrix, so the hypotheses of Corollary~\ref{rmk:multiple-Hermitian} hold.
\end{remark}

This remark has a deep implication in terms of classifying 1-abelian complex structures:

\begin{proposition}\label{prop:C(M)-Hermitian}
Let $(\mathfrak g, J_M)$ be an NLA of dimension $2n$ endowed with a 1-abelian complex structure where the matrix $M$ is regular and $C(M)$ is diagonalizable with a unique eigenvalue. Then, there exists a basis $\{\sigma^i\}_{i=1}^n$ of $\mathfrak{g}^{1,0}$ such that $J_M$ is equivalent to $J_{n,k,0}$ (see~\eqref{eq:Hermitian}) for some $k$.
\end{proposition}

\subsection{Case of singular complex matrices}
If $A$ is a singular matrix, the results in \cite{HS1} and \cite{HS2} allow us to  separate the singular part (related to blocks of Type 0 in Theorem~\ref{HS}) from the regular part and to apply the last algorithm to the regular part.  The first result is the following:
\begin{theorem}\label{singular}\cite[Lemma 2]{HS1}:
    Each square matrix $A$ is *congruent to a direct sum of the form 
    \begin{equation}\label{eq:descomposicion-singular}
    B \oplus J_{r_1}(0)\oplus ... \oplus J_{r_p}(0) \text{ with a nonsingular B},
    \end{equation}
    for some $1 \leq r_1 \leq ... \leq r_p$. This direct sum is uniquely determined up to permutation of its singular direct summands and replacement of $B$ by any matrix that is *congruent to it.
\end{theorem}

 \begin{remark}
    Note that $B$ might not exist. This is the case when $dim\, A=r_1+\dots+r_p$.
\end{remark}

In the direct sum~\eqref{eq:descomposicion-singular}, $B$ is called the \textit{regular part }of $A$ (in case of existence), and all the blocks $J_{r_k}(0)$ constitute the singular part.  Note that in general, given $r_i, r_j$, it is not true that $J_{r_i+r_j}(0)=J_{r_i} \oplus J_{r_j}(0)$.

\medskip

The process that appears in the previous theorem is called in the literature \cite{HS2}  \textit{regularization}.  The goal of this section is to describe the regularization algorithm for a singular matrix $M \in M_n(\mathbb{C})$, obtaining its regular and singular part.

\subsubsection{The regularization algorithm}\label{sec:reg-alg}
Let $M\in M_n(\mathbb C)$ be a singular matrix with rank$(M) = r<n$. The main idea of the regularization algorithm is to reduce $M$ by *congruence transformations similarly to Gauss method in order to construct a smaller matrix $M_{(1)}$ and integers $m_1$ and $m_2$ that reflect the singular part of the matrix.  The algorithm consists of 4 steps that can be iterated if needed.   

\medskip

\textbf{Step 1}: Making Gaussian elimination of rows in matrix $M$, we obtain a regular matrix $S$ (which is in fact a product of elementary matrices), such that the last $m_1:=n-r$ rows of $SM$ are zero. Then, we multiply by $S^*$ to obtain
\begin{equation}\label{step1-regularization}
    M^{(i)} = SMS^*=\begin{bmatrix}
        M'S^*\\
        0
    \end{bmatrix}=\left(\begin{array}{@{}c|c@{}}
  M'
  & N \\
\hline
  0 &
  0_{m_1}
\end{array}\right),
\end{equation}
where $0_{m_1}$ is the zero square matrix of order $m_1$.   For our purposes, we decompose the linearly independent part of $M^{(i)}$ into two blocks: a square matrix $M' \in M_{n-m_1}(\mathbb{C})$ and $N \in M_{(n-m_1)\times m_1}(\mathbb{C})$. %From this, is clear that $m_1=dim(Ker(M))$ and $rk(M)=n-m_1$. The number $m_1$ is called \textit{nullity of $M$}. 
If $N$ is the zero matrix, then 
\begin{equation}\label{N=0}
M^{(i)}=\left(\begin{array}{@{}c|c@{}}
  M'
  & 0 \\
\hline
  0 &
  0_{m_1}
\end{array}\right)
\end{equation} and therefore $M$ is *congruent to $M' \oplus J_1(0)\oplus\stackrel{m_1}{\ldots}\oplus J_1(0)$. Moreover, $M'$ is a regular matrix, otherwise the Gauss algorithm would provide new zero rows, contradicting the definition of $m_1$.  The regularization algorithm finishes here.  If $N\neq 0$, we go to Step 2.

%\begin{remark}\label{M-no-maximal}Obseve that in this case, $M^{(i)}$ is not of maximal type and therefore we will not obtain ``new'' complex structures.
%\end{remark}

\medskip

\textbf{Step 2}: The goal in this step is to simplify the matrix~$N$ as much as possible whenever $N$ is not the null matrix.  If $m_2$ denotes the rank of $N$, applying again the Gaussian elimination process to the matrix $N$, there exists a regular square matrix $R$ of order $n-m_1$, product of elementary matrices, such that the top $n-m_1-m_2$ rows of $RN$ are zero (or do not exist if $m_1+m_2 = n$) and the bottom $m_2$ rows are linearly independent:%\footnote{In the particular case $n-m_1 = 1$, $N$ is a row-matrix and $N=E$ in \eqref{step1-regularization}.}: 
\begin{equation*}\label{0-E}
    RN=\begin{bmatrix}
        0\\
        E
    \end{bmatrix}.
\end{equation*}
Here $E\in M_{m_2\times m_1}(\mathbb C)$ and rank$(E) = m_2$. Observe that $m_2 \leq m_1$.\\

Summarizing both steps 1 and 2, the following diagram follows:
\begin{eqnarray}\label{step2}
    M &\longmapsto& M^{(i)} := SMS^*=\left(\begin{array}{@{}c|c@{}}
  M'
  & N \\
\hline
  0 &
  0_{m_1}
\end{array}\right) \longmapsto (R\oplus I)SMS^*(R\oplus I)^*= \nonumber\\
   & = &(R\oplus I)\left(\begin{array}{@{}c|c@{}}
  M'
  & N \\
\hline
  0 &
  0_{m_1}
\end{array}\right)(R\oplus I)^*=\left(\begin{array}{@{}c|c@{}}
  RM'R^*
  & RN \\
\hline
  0 &
  0_{m_1}
\end{array}\right) =  \left(\begin{array}{@{}c|c@{}}
  \begin{matrix}
      M_{(1)} & B\\
      C & D
  \end{matrix}
  & \begin{matrix}
      0\\
      E
  \end{matrix} \\
\hline
  0 &
  0_{m_1}
\end{array}\right):= M^{(ii)}.
\end{eqnarray}
In the last equality, we have decomposed the block $RM'R^* \in M_{n-m_1}(\mathbb{C})$ in such a way that $D$ is a square matrix, $D\in M_{m_2}(\mathbb C)$.  So the other submatrices are of the following orders: $C \in M_{m_2 \times (n-m_2-m_1)}(\mathbb{C})$, $B \in M_{(n-m_2-m_1) \times m_2}(\mathbb{C})$ and $M_{(1)} \in M_{n-m_2-m_1}(\mathbb{C})$ (if $m_1+m_2 =n$, $M_{(1)}$ and $B$ do not exist). \\ %\begin{itemize}
    %\item %If $M_{(1)}$ is non-singular, then the algorithm ends.%
    %\item %If $M_{(1)}$ is singular. We repeat the process to the matrix $M_{(1)}$, so two more numbers are obtained: One the one hand $m_3=dim(Ker(M_{(1)}))$ and, on the other hand, $m_4$. In adittion, it also appears a new square matrix $M_{(2)}$ of order $n-m_1-m_2-m_3-m_4$. Iterating this algorithm untill we find a non-singular matrix $M_{(\tau)}$ at  step $\tau$, we will be obtaining a sequence of integer numbers $m_1 \geq m_2 \geq ... \geq m_{2\tau} \geq 0$. The most important result is that $\{m_i\}_{i=1}^{2\tau}$ are *congruence invariants and determine completely the singular part of our matrix $M$. This aspect is explained in the next section.%
%\end{itemize}

Our next objective is to simplify submatrices $E, C$ and $D$, obtaining new matrices $M^{(j)}$ which are always *congruent with the initial $M$.
%If $m_2 > 0$, then the matrix $M^{(ii)}$ obtained in step 2 (see \eqref{step2}), can be reduced to a more sparse form by *congruence by making the blocks $C$ and $D$ to be zero, as it is shown in the following two step method:

\medskip

\textbf{Step 3}: Let us focus first on the matrix $E$. Recall that $E\in M_{m_2\times m_1}(\mathbb C)$ and rank$(E) = m_2$.   We are going to construct a regular matrix $V\in GL(m_1,\mathbb C)$ such that $EV=\begin{bmatrix}
    I_{m_2} & 0
\end{bmatrix}$:

Since the rank of the block $RN$ is maximum at rows, its columns span $\mathbb{C}^{m_2}$, in particular we can assume that the first $m_2$ columns are a basis of $\mathbb{C}^{m_2}$ and the other $m_1-m_2$ columns are linear combinations of the first $m_2$ columns.  If not, there exists $V_{(1)}\in GL(m_1,\mathbb C)$, that corresponds to certain columns permutations, such that $$EV_{(1)}=      \begin{bmatrix}
          E_{m_2} & 0
      \end{bmatrix},$$
where $E_{m_2}$ is the square block made by the first $m_2$ rows and the first $m_2$ columns. Now, as the first $m_2$ columns span $\mathbb{C}^{m_2}$, $E_{m_2}$ is invertible, there exists then $E_{m_2}^{-1}$ and $$\begin{bmatrix}
          E_{m_2} & 0
      \end{bmatrix}\begin{pmatrix}
          E_{m_2}^{-1} & \\
           & 0_{m_1-m_2}
      \end{pmatrix}=\begin{bmatrix}
          I_{m_2} & 0
      \end{bmatrix}$$
So, we can define $V=V_{(1)}\begin{pmatrix}
          E_{m_2}^{-1} & \\
           & 0_{m_1-m_2}
      \end{pmatrix}$. Finally, consider then the matrix $Q:=I_{n-m_1} \oplus V^*$, we have

\begin{equation}\label{util-max-tipe}
    QM^{(ii)}Q^*=\left(\begin{array}{@{}c|c@{}}
  I_{n-m_1}
  &  \\
\hline
   &
  V^*
\end{array}\right)\left(\begin{array}{@{}c|c@{}}
  \begin{matrix}
      M_{(1)} & B\\
      C & D
  \end{matrix}
  & \begin{matrix}
      0\\
      E
  \end{matrix} \\
\hline
  0 &
  0_{m_1}
\end{array}\right)\left(\begin{array}{@{}c|c@{}}
  I_{m-m_1}
  &  \\
\hline
   &
  V
\end{array}\right)=
\end{equation}
$$=\left(\begin{array}{@{}c|c@{}}
  \begin{matrix}
      M_{(1)} & B\\
      C & D
  \end{matrix}
  & \begin{matrix}
      0\\
      E
  \end{matrix} \\
\hline
  0 &
  0_{m_1}
\end{array}\right)\left(\begin{array}{@{}c|c@{}}
  I_{m-m_1}
  &  \\
\hline
   &
  V
\end{array}\right)=\left(\begin{array}{@{}c|c@{}}
  \begin{matrix}
      M_{(1)} & B\\
      C & D
  \end{matrix}
  & \begin{matrix}
      0\\
      \begin{bmatrix}
          I_{m_2} & 0
      \end{bmatrix}
  \end{matrix} \\
\hline
  0 &
  0_{m_1}
\end{array}\right):=M^{(iii)}.$$

Observe that if $m_2<m_1$, then the last $m_1-m_2$ rows and columns are zero.  Otherwise, if $m_2=m_1$, then there are no zero columns.  This observation will be important for us in the future.

\medskip

\textbf{Step 4}: Once we have normalized matrix $E$, next goal is to make $C$ and $D$ zero, starting from $M^{(iii)}$.  This is quite easy, since the block $I_{m_2}$ spans $\mathbb{C}^{m_2}$, so we can do column transformations in order to make the blocks $C$ and $D$ vanish. This operation is given by right multiplication by an elementary matrix that performs column operations, namely the following one:
$$P=\begin{bmatrix}
    I_{n-m_1} & \\
    X & I_{m_1}
\end{bmatrix} \quad \text{where}\quad X=-\begin{bmatrix}
    C & D\\
    0 & 0
\end{bmatrix}.$$
Note that 
\begin{eqnarray}\label{reducedform}
    P^*M^{(iii)}P&=&\left(\begin{array}{@{}c|c@{}}
  I_{n-m_1}
  & X^* \\
\hline
   &
  I_{m_1}
\end{array}\right)\left(\begin{array}{@{}c|c@{}}
  \begin{matrix}
      M_{(1)} & B\\
      C & D
  \end{matrix}
  & \begin{matrix}
      0\\
      \begin{bmatrix}
          I_{m_2} & 0
      \end{bmatrix}
  \end{matrix} \\
\hline
  0 &
  0_{m_1}
\end{array}\right)\left(\begin{array}{@{}c|c@{}}
  I_{n-m_1}
  &  \\
\hline
  X &
  I_{m_1}
\end{array}\right)= \nonumber \\
&=&\left(\begin{array}{@{}c|c@{}}
  I_{n-m_1}
  &  X^*\\
\hline
   &
  I_{m_1}
\end{array}\right)\left(\begin{array}{@{}c|c@{}}
  \begin{matrix}
      M_{(1)} & B\\
      0 & 0
  \end{matrix}
  & \begin{matrix}
      0\\
      \begin{bmatrix}
          I_{m_2} & 0
      \end{bmatrix}
  \end{matrix} \\
\hline
  0 &
  0_{m_1}
\end{array}\right)=\left(\begin{array}{@{}c|c@{}}
  \begin{matrix}
      M_{(1)} & B\\
      0 & 0_{m_2}
  \end{matrix}
  & \begin{matrix}
      0\\
      \begin{bmatrix}
          I_{m_2} & 0
      \end{bmatrix}
  \end{matrix} \\
\hline
  0 &
  0_{m_1}
\end{array}\right):=M^{(iv)}.
\end{eqnarray}

Observe that, in the process of multiplying $M^{(iii)}P$ , we are modifying the first $n-m_1$ columns of $M^{(iii)}$ by adding a linear combination determined by the transformation matrix $P$. Then, when multiplying the result on the left by $P^*$, we are modifying the first $n-m_1$ rows by adding to them linear combinations of the last $m_1$ rows. However, the last $m_1$ rows of $M^{(iii)}$ are zero; hence, left multiplication by $P^*$ does not alter the matrix $M^{(iii)}P$.

\medskip

At this point, we have finished the 4 steps in the first iteration of the algorithm.  Observe that as results we have obtained four elements: $M_{(1)}\in M_{n-m_1-m_2}(\mathbb C)$, $B\in M_{(n-m_1-m_2)\times m_2}(\mathbb C)$, $m_1 = n-rank(M)$, $m_2 = rank(N)$. At this point, it is convenient to analyze how many possibilities for $M_{(1)}$ can we have:

\medskip

Case 1:  If $m_1+m_2=n$, then $M_{(1)}$ is a zero-by-zero matrix, i.e. it does not exist, so $M$ is *congruent to \begin{equation}\label{m1+m2=n}
        M^{(iv)}=\begin{pmatrix}
        0_{m_2} & \begin{bmatrix}
          I_{m_2} & 0
      \end{bmatrix}\\
      0 & 0_{m_1}
    \end{pmatrix},
    \end{equation}
    %Look that $rk(M''')=m-m_1=m_2$ and $(M''')^2=0$. The Jordan canonical form of $M'''$ has $m_2$ Jordan blocks of order 2 and $m_1-m_2$ Jordan blocks of order 1.%
    and the algorithm finishes.    

 \medskip
 
Case 2\label{caso2}:  If $m_1+m_2 < n$ and $M_{(1)}=0_{n-m_1-m_2}$. Denote $m_3:=n-m_1-m_2$, then $$M^{(iv)}=\left(\begin{array}{@{}c|c@{}}
  \begin{matrix}
      0_{m_3} & B\\
      0 & 0_{m_2}
  \end{matrix}
  & \begin{matrix}
      0\\
      \begin{bmatrix}
          I_{m_2} & 0
      \end{bmatrix}
  \end{matrix} \\
\hline
  0 &
  0_{m_1}
\end{array}\right)$$
Notice that $B$ has $m_3$ rows and $m_2$ columns and it has full row rank. Once again, using an argument analogous to that in Step 3, there exists a matrix $W$ such that $BW=\begin{bmatrix}
    I_{m_3} & 0
\end{bmatrix}$. Considering the matrix $S=I_{m_3}\oplus V \oplus I_{m_1}$, then $$S^*M^{(iv)}S=\left(\begin{array}{@{}c|c@{}}
  \begin{matrix}
      0_{m_3} & \begin{bmatrix}
          I_{m_3} & 0
      \end{bmatrix}\\
      0 & 0_{m_2}
  \end{matrix}
  & \begin{matrix}
      0\\
      \begin{bmatrix}
          V^* & 0
      \end{bmatrix}
  \end{matrix} \\
\hline
  0 &
  0_{m_1}
\end{array}\right).$$
In order to reestablish the $\begin{bmatrix}
    I_{m_2} & 0
\end{bmatrix}$ block, take the matrix $R=I_{m_3+m_2} \oplus (V^*)^{-1} \oplus I_{m_1}$. Combining the above, and naming $P:=R^*S^*$ it follows that 
\begin{eqnarray}\label{A=0}
    M^{(v)}_2&:=&R^*(S^*M^{(iv)}S)R=PM^{(iv)}P^* \nonumber \\
   & =&\left(\begin{array}{@{}c|c@{}}
  \begin{matrix}
      0_{m_3} & \begin{bmatrix}
          I_{m_3} & 0
      \end{bmatrix}\\
      0 & 0_{m_2}
  \end{matrix}
  & \begin{matrix}
      0\\
      \begin{bmatrix}
          I_{m_2} & 0
      \end{bmatrix}
  \end{matrix} \\
\hline
  0 &
  0_{m_1}
\end{array}\right).
\end{eqnarray}

This concludes the algorithm for this case.   Note that matrices of form~\eqref{A=0} are quite similar to the ones of form~\eqref{m1+m2=n} and in the future we will treat this two cases together. 

\medskip

Case 3: If $m_1+m_2 < n$ and $M_{(1)}$ is a regular matrix, then the columns of $M^{(iv)}$ span $\mathbb{C}^{n-m_1-m_2}$, so it is possible to annihilate block $B$ by adding linear combinations of the columns of $M_{(1)}$. For this aim, we shall consider the matrix $$S=\begin{pmatrix}
        I_{n-m_1-m_2} & -M_{(1)}^{-1}B & 0\\
        0 & I_{m_2} & 0\\
        0 & 0 & I_{m_1}
    \end{pmatrix}.$$
   The product $M^{(iv)}S$ is of the form \eqref{reducedform} but with $B=0$. However, when we multiply by $S^*$ by the left, we get $$S^*M^{(iv)}S=\begin{pmatrix}
        M_{(1)} & 0 & 0\\
        -B(M_{(1)}^{-1})^*M_{(1)} & 0_{m_2} & \begin{bmatrix}
            I_{m_2} & 0
        \end{bmatrix}\\
        0 & 0 & 0_{m_1}
    \end{pmatrix}.$$
    This matrix is of form \eqref{step2}, so we can apply Step 4 in order to obtain
    \begin{equation}\label{B=0}
        M_3^{(v)}:=\begin{pmatrix}
            M_{(1)} & 0 & 0\\
            0 & 0_{m_2} & \begin{bmatrix}
                I_{m_2} & 0
            \end{bmatrix}\\
            0 & 0 & 0_{m_1}
        \end{pmatrix}.
    \end{equation}  

    Since $M_{(1)}$ is a regular matrix, the algorithm finishes.   

\medskip

Case 4: If $m_1+m_2 < n$ and $M_{(1)}$ is singular non-zero matrix, then, we repeat the whole process to the matrix $M_{(1)}$, starting with $M^{(iv)}$ as in \eqref{reducedform}. After the 4 steps, we will obtain two numbers $m_3$ and $m_4$ and a matrix $M_{(2)} \in M_{n-m_1-m_2-m_3-m_4}(\mathbb{C})$ such that $M$ is *congruent to

\begin{equation}\label{finalII}
    M^{(vi)}:=\begin{pmatrix}
        M_{(2)} & B' & 0 & 0 & 0\\
      0 & 0_{m_4} & \begin{bmatrix}
          I_{m_4} & 0
      \end{bmatrix} & 0 & 0\\
      0 & 0 & 0_{m_3} & \begin{bmatrix}
          I_{m_3} & 0
      \end{bmatrix} & 0\\
      0 & 0 & 0 & 0_{m_2} & \begin{bmatrix}
          I_{m_2} & 0
      \end{bmatrix}\\
      0 & 0 & 0 & 0 & 0_{m_1}
    \end{pmatrix}.
\end{equation}

\medskip

Iterating this algorithm until we find a non-singular matrix $M_{(\tau)}$ at step $\tau$, we will get a sequence of integer numbers $m_1 \geq m_2 \geq ... \geq m_{2\tau} \geq 0$. The most important result is that $\{m_i\}_{i=1}^{2\tau}$ are *congruence invariants and determine completely the singular part of the initial matrix $M$ (see~\cite{HS2} for details). 

\begin{definition}\label{reducedcong}
    Given a non-regular square matrix $M \in M_n(\mathbb{C)}$, a matrix $\Tilde{M}$ is called \textit{*congruence reduced form} of $M$ if it is of form \eqref{reducedform}, where the block $M_{(1)}$ does not exist, is the zero matrix, or is a regular matrix and, whenever $M_{(1)}$ exists, $\begin{bmatrix}
    M_{(1)} & B
\end{bmatrix}$ has the maximum linearly independent rows, that is $n-m_1-m_2$.
\end{definition}

The next example shows in detail how the algorithm works:

\begin{example}
Let us apply the previous algorithm to the singular matrix $M=\begin{pmatrix}
    1 & 2 & i\\
    4 & 0 & -2i\\
    -3 & 2 & 3i
\end{pmatrix}.$
Starting with Step 1, we apply Gauss method in order to reduce $M$, obtaining a matrix $S$ that is the product of elementary matrices. In this case $S=\begin{pmatrix}
    1 & 0 & 0\\
    -4 & 1 & 0\\
    -1 & 1 & 1
\end{pmatrix}$ and we have that $$SMS^*=\begin{pmatrix}
    1 & -2 & 1+i\\
    0 & -8 & -8-6i\\
    0 & 0 & 0
\end{pmatrix}.$$
Observe that $m_1=1$. Here, we have $M'=\begin{pmatrix}
    1 & -2\\
    0 & -8
\end{pmatrix} \in M_2(\mathbb{C})$ and $N=\begin{pmatrix}
    1+i\\
    -8-6i
\end{pmatrix} \in M_{2 \times 1}(\mathbb{C})$.  Since $N\neq 0$, we go to
Step 2, applying again Gauss method to $N$ in order to obtain zeros in the top rows of $N$ getting the block $\begin{pmatrix}
    0\\
    E
\end{pmatrix}$. Again, the method gives us a matrix $R \in GL_2(\mathbb{C})$ product of elementary ones. In this case, $R=\begin{pmatrix}
    8+6i & 1+i\\
    0 & 1
\end{pmatrix}$ and $RN=\begin{pmatrix}
    0\\
    -8-6i
\end{pmatrix}$. From this, $E=\begin{pmatrix}
    -8-6i
\end{pmatrix}$ and $m_2=rk(N)=1$. Consider now $T=R\oplus I_1$, then $$M''=TSMS^*T^*=\begin{pmatrix}
    56+4i & -24-20i & 0\\
    -8+8i & -8 & -8-6i\\
    0 & 0 & 0
\end{pmatrix}.$$ 
Attending to the decomposition of $RM'R^*$ mentioned at the end of Step 2, we have that that 
$$M_{(1)}=(56+4i),\quad B=(-24-20i),\quad C=(-8+8i),\quad D=(-8).$$\\
As $m_2=1>0$, we can go to Steps 3 and 4, in order to have $C=D=0$. Let us apply first step 3 in order to reduce $E$ to $I_1$ by multiplying by a $1 \times 1$-matrix $V$. In this case  $V=\begin{pmatrix}
    \frac{1}{-8-6i}
\end{pmatrix}$ and if we consider $Q=\begin{pmatrix}
    1 & & \\
     & 1 & \\
     &   & \frac{1}{-8+6i}
\end{pmatrix}$, then $$M'''=QM''Q^*=\begin{pmatrix}
    56+4i & -24-20i & 0\\
    -8+8i & -8 & 1\\
    0 & 0 & 0
\end{pmatrix}.$$
Let us now execute the Step 4, that consists in making zeros in the blocks $\begin{pmatrix}
    -8
\end{pmatrix}$ and $\begin{pmatrix}
    -8+8i
\end{pmatrix}$. In this case, it is straightforward to check that adding $8-8i$ times the third column to the first column, and adding $8$ times the third column to the second column, we get the blocks $C=D=0$. That is, consider the matrix $$P=\begin{pmatrix}
    1 & 0 & 0\\
    0 & 1 & 0\\
    8-8i & 8 & 1
\end{pmatrix},$$
and \begin{equation*}
    P^*M'''P=\begin{pmatrix}
    56+4i & -24-20i & 0\\
    0 & 0 & 1\\
    0 & 0 & 0
\end{pmatrix}:=M^{(iv)}.
\end{equation*}
In the notation of the algorithm, here $X=-\begin{pmatrix}
    C & D
\end{pmatrix}$ and $P$ is defined above.\\
So, finally we arrived to the *congruence reduced form, $M^{(iv)}$, of $M$ (see Definition \ref{reducedcong}). Let us now determine which of the four cases described above we are in after Step 4. By the case, $m_1+m_2=2 < 3=n$ and $M_{(1)}=\begin{pmatrix}
    56+4i
\end{pmatrix}$ that is non-singular, so we are in case 3. We can make a zero in block $B=\begin{pmatrix}
    -24-20i
\end{pmatrix}$ by adding the second column the first one multiplied by $\frac{24+20i}{56+4i}=\frac{89}{197}+\frac{64}{197}i$. Considering the matrix $$S_1=\begin{pmatrix}
    1 & \frac{89}{197}+\frac{64}{197}i & 0\\
    0 & 1 & 0\\
    0 & 0 & 1
\end{pmatrix},$$ we obtain $$S_1^*M^{(iv)}S_1=\begin{pmatrix}
    56+4i & 0 & 0\\
    \frac{5240}{197}-\frac{3228}{197}i & 0 & 1\\
    0 & 0 & 0
\end{pmatrix}.$$
Again, we have to restablish a zero in the position of matrix $B$. For this, we add $-\left( \frac{5240}{197}-\frac{3228}{197}i \right)$ times the third column to the first one, by multiplying by the right side by the matrix $$P_1=\begin{pmatrix}
    1 & 0 & 0\\
    0 & 1 & 0\\
    - \frac{5240}{197}+\frac{3228}{197}i & 0 & 1
\end{pmatrix}$$ and finally we get $$P_1^*S_1^*M^{(iv)}S_1P_1=\begin{pmatrix}
    56+4i & & \\
     & 0 & 1\\
     & & 0
\end{pmatrix}=M_3^{(v)}.$$
By combining all the change matrices obtained at each step, we can conclude that $$(P_1^*S_1^*P^*QTS)M(S^*T^*Q^*PS_1P_1)=M_3^{(v)}.$$

%in order to achieve the *congruence reduced form of $M$.

%$$M\sim M_{(1)}\oplus J_1^{[m_1-m_2]}(0)\oplus J_2^{[m_2]}(0)=M_{(1)}\oplus J_2(0)=\begin{pmatrix}
    %56-44i & & \\
     %& 0 & 1\\
     %& & 0
%\end{pmatrix}$$

\end{example}

\medskip

In the algorithm described above, given a matrix $M$ we construct its reduced *congruence form. This reduced *congruence form yields  a sequence of integers $m_1 \geq m_2 \geq ... \geq m_{2\tau} \geq 0$ and nonsingular square matrix $M_{(\tau)}$ of dimension $dim(M_{(\tau)})= dim \,M - \sum_{i=1}^{2\tau} m_i$. Note that $M_{(\tau)}$ does not exist if $dim\,M=\sum_{i=1}^{2\tau}m_i$.

\medskip

We finish this section setting the relation between the matrix $M_{(\tau)}$ and the announced decomposition~\eqref{eq:descomposicion-singular}. The details for the proof can be found in \cite[Theorem~6]{HS2}.

\begin{theorem}\label{important}
Let $M \in M_n(\mathbb{C})$ be a square singular matrix.  Then there exist integers $\tau, m_1,...,m_{2\tau}$ such that the *congruence reduced form or $M$ has a block $M_{(\tau)}$ of size $n-\sum_{i=1}^{2\tau} m_i$ and: 
\begin{itemize}
\item[a)] $M_{(\tau)}$ is regular or it does not exist if $n=\sum_{i=1}^{2\tau} m_i$.
    \item[b)] The integers $\tau, m_1,...,m_{2\tau}$ and the *congruence class of $M_{(\tau)}$ are *congruence invariants of $M$.
    \item[c)] $M$ is *congruent to $M_{(\tau)} \oplus Q$, in which 
    \begin{equation*}\label{Q}
        Q=\begin{bmatrix}
            0_{m_{2\tau}} & \begin{bmatrix}
                I_{m_{2\tau}} & 0
            \end{bmatrix} & & &  &   \\
             & 0_{m_{2\tau-1}} & \begin{bmatrix}
                I_{m_{2\tau-1}} & 0
            \end{bmatrix} & & & \\
             & & \ddots & \ddots & & \\
              & & & 0_{m_3} & \begin{bmatrix}
                I_{m_3} & 0
            \end{bmatrix} & \\
             & & & & 0_{m_2} & \begin{bmatrix}
                I_{m_2} & 0
            \end{bmatrix}\\
             & & & & & 0_{m_1}
        \end{bmatrix}.
    \end{equation*}
    \item[d)] $M$ is *congruent to $M_{(\tau)} \oplus N$ in which $N$ is of the following form \begin{equation*}
    N=J_1^{[m_1-m_2]}(0) \oplus J_2^{[m_2-m_3]}(0) \oplus J_3^{[m_3-m_4]}(0) \oplus ... \oplus J_{2\tau -1}^{[m_{2\tau -1}-m_{2\tau}]}(0) \oplus J_{2\tau}^{[m_{2\tau}]}(0),
    \end{equation*}
    where $J_n^{[m]}(0)=J_n(0)\oplus ... \oplus J_n(0)$ ($m$ times).
\end{itemize}
\end{theorem}
%In the next lines, a sketch of the proof is given 

\medskip

The proof of part d) of the previous theorem is explained in detail in \cite{HS2}. It can be divided in two steps: the first one consists on showing that $N$ is the Jordan canonical form of $Q$. The second one  proves that the directed graphs of $Q$ and $N$ are isomorphic, and therefore it is possible to conclude that the Jordan canonical form of $Q$ can be achieved via a permutation similarity, which is a *congruence invariant.

%\textbf{Example 1} [HS, Example 7] Consider the matrix $$M=\begin{pmatrix}
 %   1 & 2\\
  %  0 & 1
%\end{pmatrix}$$ 

%%%%%%%%%%%%%%%%%%%%%%%%%%%%%%%%%%%%%%%%%%%%%%%%%%%%%%%%%%%%%%%%%%%%%

\section{Classification of 1-abelian complex structures on NLAs}\label{sec:4}
In this section we study the two problems about complex structures in the setting of 1-abelian complex structures: its classification up to equivalence using the previous analysis of *congruence,  and the identification of the underlying Lie algebras. We focus on dimensions 4, 6 recovering the already known results (Sections~\ref{sub:4}, \ref{sub:6} respectively), we  completely analyze the 8-dimensional case and 8 (Section~\ref{sub:8}) and give some interesting results for the general case (Section~\ref{sub:general}).
%\begin{remark}
%    Let $\mathfrak{g}$ be a $2n$-dimensional nilpotent Lie algebra and $M \in M_{n-1}(\mathbb{C})$, we denote by $J_M$ the abelian complex structure on $\mathfrak{g}$ given by \begin{equation}
 %       \begin{cases}
 %           d\omega^1=\dots=d\omega^{n-1}=0\\
%            d\omega^n=\omega M\bar\omega^T
 %       \end{cases},
%    \end{equation}
%where $\omega=\begin{pmatrix}
%    \omega^1 & \dots & \omega^n
%\end{pmatrix}$.
%\end{remark}

%According to Definition \ref{equivalence}, the following lemma is trivially verified
%\begin{lemma}
%    If $M$ is *congruent to $M'$, then $J_M \sim J_{M'}$.
%\end{lemma}

As a first approximation to a classification of 1-abelian complex structures up to equivalence, we have the following result, where we have adapted Theorem~\ref{important} and we are using that *congruent matrices provide equivalent complex structures:

\begin{theorem}\label{class}
    Let ($\mathfrak{g},J$) be a $2n$-dimensional nilpotent Lie algebra endowed with a $1$-abelian complex structure $J$ \eqref{ouracs}. Then, J is equivalent to 
    \begin{equation*}
    \begin{cases}
            d\sigma^1=...=d\sigma^{n-1}=0,\\
            d\sigma^n=\sigma Q \overline{\sigma}^T,
        \end{cases}
    \end{equation*}
    where $Q\in M_{n-1}(\mathbb C)$ splits as $Q= R\oplus S$ where: \begin{itemize}
    \item $R = \bigoplus_{i=1}^q H_{2k_i}(\mu_i)\oplus \bigoplus_{j=1}^p\varepsilon_je^{\nicefrac{i\phi_j}{2}}\Delta_{n_j}$;\\[-2pt]
        \item $S=J_1^{[m_1-m_2]}(0) \oplus J_2^{[m_2-m_3]}(0) \oplus J_3^{[m_3-m_4]}(0) \oplus ... \oplus J_{2\tau -1}^{[m_{2\tau -1}-m_{2\tau}]}(0) \oplus J_{2\tau}^{[m_{2\tau}]}(0)$ for some $m_1 \geq m_2 \geq ... \geq m_{2\tau} \geq 0$;\\[-2pt]
        \item $2\displaystyle\sum_{i=1}^q k_i + \sum_{j=1}^pn_j + \sum_{s=1}^{2\tau}m_s =n-1$.
    \end{itemize} 
\end{theorem}

The previous theorem gives rise to the following definition:

\begin{definition}
Let $Q = R\oplus S$ be as in the previous theorem.  We will say that $R$ is the \emph{regular part of $Q$} and $S$ is the \emph{singular part of $Q$}.
\end{definition}

It is worth noticing that if $Q$ is a regular matrix, then the term $S$ does not appear (all $m_i=0$) and $Q=R$.  However, if $Q$ is singular, the term $R$ could appear in the decomposition of $Q$, since this term corresponds to the matrix $M_{(\tau)}$ of the regularization algorithm.  This observation directs us to the next notion:

\begin{definition}\label{def:completely-singular}
Let $Q = R\oplus S$ be a singular matrix splitted as in Theorem~\ref{class}.   We will say that \emph{$Q$ is completely singular} if its regular part $R$ does not appear in the decomposition.  
\end{definition}

\medskip

However, two 1-abelian complex structures that are equivalent may not be represented by *congruent matrices, as the next example shows:
\begin{example}
    Consider the complex structures $J_M$ and $J_{M'}$ where 
    $$M=\begin{pmatrix}
        0 & 1\\
        2i & 0
    \end{pmatrix},\quad M'=\begin{pmatrix}
        0 & 1\\
        2 & 0
    \end{pmatrix}.$$
Note that $J_M \sim J_{M'}$, indeed, take $$c=e^{-i\frac{\pi}{4}}; \quad P=\begin{pmatrix}
    e^{i\frac{\pi}{4}} & 0\\
    0 & 1
\end{pmatrix},$$ but $M$ is not *congruent to $M'$ because $$C(M)=\begin{pmatrix}
    2 & 0\\
    0 & \nicefrac{1}{2}
\end{pmatrix}\nsim_{sim}C(M')=\begin{pmatrix}
    2i & 0\\
    0 & \frac{1}{2}i
\end{pmatrix}$$that is, their *cosquares are not similar, because they have distinct eigenvalues. However, \emph{the type} of the eigenvalues of $C(M)$ and $C(M')$ is the same in the sense that each matrix has two different eigenvalues of modulus distinct of $1$.
\end{example}
The behavior that we have shown in the previous example can be generalized: if $J_M \sim J_{M'}$, the matrices $M$ and $M'$ have similar properties in terms of the decomposition given in Theorem~\ref{class}.
\begin{lemma}
    Let $M, M' \in M_n(\mathbb{C})$ such that $J_M \sim J_{M'}$. Then the number of blocks $H_{2i}(\mu), \, \Delta_j$ and $J_k(0)$ in Theorem~\ref{class} are the same for $M$ and for $M'$.
\end{lemma}
\begin{proof}
    If $J_M \sim J_{M'}$, then there exist $c \in \mathbb{C} \setminus \{0\}$ and $P \in GL(n-1, \mathbb{C})$ such that $M'=c(PMP^*)$, so $M$ is *congruent to $\frac{1}{c}M'$ and by Lemma~\ref{cosquaresim}, $C(M) \sim_{sim} C(\frac{1}{c}M')$.  Since $C(\lambda M) = \frac{\lambda}{\bar \lambda} C(M)$ for any $\lambda\in\mathbb C^*$, we obtain that $$Sp(C(M))=Sp\biggl( C\biggl(\frac{1}{c}M'\biggl)\biggl)=\frac{\bar c}{c}Sp(C(M')).$$If $(\mu, \bar \mu^{-1})$ is a pair of eigenvalues of modulus $>1$ of $C(M)$ (so they give a block $H_{2k}(\mu)$ in the *congruence representative of $M$), then $\left( \nicefrac{c\mu}{\bar c}, \nicefrac{c\bar \mu^{-1}}{\bar c}\right)$ is another such pair for $C(M')$ and they give another block $H_{2k}(\nicefrac{c\mu}{\bar c})$ in the *congruence representative of $M'$ (note that $|\mu|=|\nicefrac{c\mu}{\bar c}|$). Also, if $z$ is an eigenvalue of $C(M)$ with $|z|=1$, then $\nicefrac{cz}{\bar c}$ is an eigenvalue of $C(M')$ of modulus 1.  Finally, the number of blocks of the singular part is a *congruence invariant since all the numbers $m_j$ are.
\end{proof}

%{\color{red}
In what follows we classify $2n$ NLAs endowed with 1-abelian complex structures when $n\leq 4$.  We are interested in 1-abelian complex structures $J_M$ that are not trivial extensions of $J_{N}$ with $N$ a matrix of lower dimension of $M$.  Observe that if $M = N\oplus 0_m$ and $\{\omega^1,\ldots, \omega^n\}$ is an adapted basis, then $d\omega^n\in\langle\omega^1,\ldots, \omega^{(n-1)-m}\rangle\wedge \langle\overline\omega^1,\ldots, \overline\omega^{(n-1)-m}\rangle$.  Taking into account Theorem~\ref{class} and the description of the *congruence blocks, this condition can only happen if $m_1-m_2\neq 0$.

This discussion leads to the following definition:

\begin{definition}\label{def:maximal-type}
Let $(\mathfrak{g}, J_M)$ be 1-abelian complex structure on a $2n$-dimensional NLA where $M$ is a singular matrix. We will say that $J_M$ is \emph{of maximal type} if $m_1 = m_2$, being $m_1$ and $m_2$ the *congruence invariants defined in the regularization algorithm.
\end{definition}

%We can check very quickly if a matrix $M$ is not of maximal type:  suppose that the matrix $M$ has a zero row, let us say, the $i$-row.  Then, if the  corresponding $i$-column is also zero, then $M$ is not of maximal type, since the row and column permutation matrices induce a *congruence with $M' \oplus 0_1$ as it is shown in the following example.  
%\begin{example}
%    Consider the complex structure $J_M$ where $M\in M_4(\mathbb C)$ given by: $$M=\begin{pmatrix}
%        4 & 0 & i & 2\\
%        0 & 0 & 0 & 0\\
 %       6 & 0 & -1 & -1\\
 %       -5+i & 0 & 7 & 1
%    \end{pmatrix}$$
%    then for $$P:=P_{24}=\begin{pmatrix}
 %       1 & & & \\
 %        &  & & 1\\
%         & & 1 & \\
 %        &1 & & 
%    \end{pmatrix}$$
%    we have that $$PMP^*=\left(\begin{array}{@{}c|c@{}}
 % \begin{matrix}
%      4 & 2 & i\\
%    -5+i & 1 & 7\\
%    6 & -1 & -1
%  \end{matrix}
 % & \begin{matrix}
%      0
%  \end{matrix} \\
%\hline
%  0 &
%  0
%\end{array}\right),$$ and therefore $M$ is not of maximal type.
%\end{example}

%Taking into account the complex equations~\eqref{ouracs}, $J_M$ is not of maximal type if and only if there exists an index $i\in\{1,\ldots, n-1\}$ such that $\omega^i$ and $\omega^{\bar i}$ do not appear in the expression of $d\omega^n$.

%\begin{example}
%For instance $d\omega^1 = d\omega^2 = 0,\,  d\omega^3 =\omega^{1\bar2}$ is a complex structure of maximal type.  However, $d\omega^1 = d\omega^2 =  0,\, d\omega^3 =\omega^{1\bar1}$ is not.
%\end{example} 

%The previous example shows that $J_M$ can be of maximal type even if the matrix $M$ is singular.  
There exists a lower bound for the rank of a squared matrix $M$ to provide a complex structures of maximal type:

\begin{lemma}\label{lemma-rango-maximal}
Let $M\in M_n(\mathbb C)$ be a singular matrix of rank $(M) = r<n$.  If $J_M$  is of maximal type, then  $r\geq \lfloor\frac{n+1}{2}\rfloor.$    
\end{lemma}

\begin{proof}
Recall that by definition, $m_1= n-r$.  In order to apply the regularization algorithm we need that $m_1+m_2\leq n$, where $m_2\leq m_1$.  Imposing $m_1=m_2$ we obtain that $2n-2r\leq n$ or equivalently, $n\leq 2r$.
%According to the previous discussion, the minimum number of summands for $d\omega^{n+1}$ for a $J_M$ of maximal type is precisely $\lfloor\frac{n+1}{2}\rfloor$, corresponding for the following simple expressions:
%$$d\omega^{n+1} = \sum_{j=1}^{k} A_{(2j-1)2j}\omega^{(2j-1)\,\overline{2j}},\quad \text{ if }n = 2k,$$ or 
%$$d\omega^{n+1} = \sum_{j=1}^{k} A_{(2j-1)2j}\omega^{(2j-1)\,\overline{2j}} + A_{(n-1)n}\,\omega^{(n-1)\bar n},\quad \text{ if }n = 2k+1.$$  Observe that the associated matrices have rank $=\lfloor\frac{n+1}{2}\rfloor$.
\end{proof}

%Analyzing the regularization algorithm from this perspective, we obtain an important obstruction to the value of invariant $m_2$:

%\begin{proposition}\label{prop:m1m2}
%Let $M\in M_n(\mathbb C)$ be a singular matrix and consider the 1-abelian complex structure $J_M$ associated to $M$.  Then,
%$$M \text{ is of maximal rank} \Longleftrightarrow m_2 = m_1.$$
%\end{proposition}

%\begin{proof}
%It follows directly from equations~\eqref{N=0} and \eqref{util-max-tipe}.  Recall that $m_2 = $ rank $N$ in~\eqref{step1-regularization}.  Observe that in the next steps of the algorithm is not possible to obtain new zero rows.
%\end{proof}

\begin{remark}\label{rmk:singular-m}
Observe that in the case of structures of maximal type, the matrix $M_{(1)}$ in \eqref{reducedform} has dimension $2r-n$.  It and can be a regular matrix, non-regular or even does not exist in the particular case $2r = n$.
\end{remark}

\begin{remark}\label{Hermitian-max}
Hermitian matrices are diagonal under the *congruence.  That means that a Hermitian matrix provides a complex structure of maximal type if and only if it is regular.
\end{remark}

%}

\medskip

Next, we classify 1-abelian complex structures of maximal type in dimensions 4 and 6 recovering already known results and provide a new classification for dimension 8 identifying the underlying Lie algebras. We obtain a list of nine NLAs of dimension~8 admitting this type of structures (see Table~\ref{Tabla-dim8} for details).  We also set new results for the general case, in arbitrary dimensions.
%and set new results for the general case.

%%%%%%%%%%%%%%%%%%%%%%%%%%%%%%%%%%%%%%%

\subsection{The 4-dimensional case}\label{sub:4}

Consider $(\mathfrak g, J_M)$ given by~\eqref{ouracs} where $M\in M_{1}(\mathbb C)$, i.e., $M =(A)$ is simply a complex number and $d\omega^2 = A\omega^{1\bar1}$.  Clearly:
\begin{itemize}
\item[(i)] $M$ is Hermitian if and only if $A\in \mathbb R$;
\item[(ii)] $M$ is regular if and only if $A\neq 0$;
\item[(iii)] $M$ is singular if and only if $A=0$.
\end{itemize}

In this situation, for cases $(i)$ and $(ii)$ it is possible to rescale the 1-form $\omega^2$ in order to obtain very simplified complex equations, as it is already known.  Define a new basis \{$\sigma^1 = \omega^1,\, \sigma^2 = \frac1A\omega^2$\} and therefore:
$$d\sigma^1 = 0,\quad d\sigma^2 = \sigma^{1\bar1}.$$
Finally, we can conclude that there exist only two 1-abelian complex structures on NLAs of dimension 4:
$$d\sigma^1 = 0,\quad d\sigma^2 = \varepsilon\,\sigma^{1\bar1},\quad \varepsilon =\{0, 1\}.$$

%%%%%%%%%%%%%%%%%%%%%%%%%%%%%%%%%%%%%%%%%%%%%%%%%%%%%%%%%

\subsection{The 6-dimensional case}\label{sub:6}

Let $(\mathfrak g, J_M)$ be a 1-abelian complex structure on an NLA of dimension 6 given by: 
\begin{equation*}\label{cseq}
    \begin{cases}
            d\omega^1=d\omega^2=0,\\
            d\omega^3= \begin{pmatrix}
        \omega^1 & \omega^2
    \end{pmatrix}
    M
    \begin{pmatrix}
        \omega^{\overline{1}}\\
        \omega^{\overline{2}}
    \end{pmatrix},
    \end{cases}\qquad \text{where}\,\, M\in M_2(\mathbb C).
\end{equation*}

The first classification of general abelian complex structures on Lie algebras of dimension 6 is due to Andrada, Barberis and Dotti in \cite{ABD} and then it was reformulated by Ceballos, Otal, Ugarte and Villacampa in \cite{COUV} for NLAs.  We will recover the classification of those being 1-abelian using this new method.  

\medskip
Let us begin with the case when $M$ is singular with $r= rank(M)=1$.  We need to apply the regularization algorithm (see Section~\ref{sec:reg-alg}) to the matrix $M$.  %In order to obtain a matrix $M^{(iii)}$ of maximal type, we need $m_2 = m_1 = n-rank(M) = 1$.  Since $2r = n$, 
According to Remark~\ref{rmk:singular-m}, since $2r = n$, $M_{(1)}$ does not exist, i.e, $M$ is completely singular and 
$$M^{(iii)}=\begin{pmatrix}
0&1\\0&0
\end{pmatrix} = J_2(0)$$ and there only exists one complex structure of this type:

\begin{theorem}\label{6-singular}
    Let $(\mathfrak{g}, J_M)$ be a $6$-dimensional NLA with a $1$-abelian complex structure where $M$ is a singular matrix of rank $1$ and maximal type. Then there exists a basis $\{\sigma^i\}_{i=1}^3$ of $\mathfrak{g}^{1,0}$ such that the complex structure equations are:
        $$J^3_{s1} : \begin{cases}
            d\sigma^1=d\sigma^2=0,\quad  d\sigma^3=\sigma^{1\bar 2}.\\
            %d\sigma^3=\sigma^{1\bar 2}.
        \end{cases}$$
\end{theorem}

Suppose now that $M$ is regular and Hermitian.  It suffices to apply Theorem~\ref{teo:Hermitian} where $r=0$ (since $M$ is regular), $n=3$ and $k=1$ or $2$:

\begin{theorem}\label{6-Hermitian}
       Let $(\mathfrak{g}, J_M)$ be a $6$-dimensional NLA with a $1$-abelian complex structure where $M$ is Hermitian of full rank. Then there exists a basis $\{\sigma^i\}_{i=1}^3$ of $\mathfrak{g}^{1,0}$ such that the complex structure equations are one of the following: 
       $$J_{3,2,0}:\begin{cases}
            d\sigma^1=d\sigma^2=0,\\
            d\sigma^3=\sigma^{1\bar 1}+\sigma^{2\bar 2};
        \end{cases} \quad J_{3,1,0} : \begin{cases}
            d\sigma^1=d\sigma^2=0,\\
            d\sigma^3=\sigma^{1\bar 1}-\sigma^{2\bar 2}.
        \end{cases}$$
    \end{theorem}

\medskip

Finally, if $M$ is regular but not Hermitian, then $S=0$ in Theorem~\ref{class} and the number of blocks $\Delta_j$ and $H_{2i}(\mu)$ is determined by the similarity type of its *cosquare, $C(M)$. In this case $C(M)\in GL(2,\mathbb C)$. Using the spectral properties of $C(M)$ (Lemma \ref{eigenvalue}), the possibilities for the block decomposition of $M$ are the following: 

\begin{lemma}\label{lema: block-dim6}
Let $(\mathfrak{g}, J_M)$ be a $6$-dimensional NLA with a $1$-abelian complex structure where $M$ is a regular non-Hermitian matrix.  Then, there exists a basis $\{\tau^i\}_{i=1}^3$ of $\mathfrak{g}^{1,0}$ for which matrix $M$ is one of the following, where $\varepsilon, \varepsilon_i = \pm 1$:
\begin{itemize}
\item[(i)] $M= c\widetilde M$, where $|c| = 1$ and $\widetilde M$ is Hermitian.
\item[(ii)] $M=\varepsilon e^{i\nicefrac{\theta}{2}}\Delta_2=\begin{pmatrix}
        & \varepsilon e^{i\frac{\theta}{2}}\\
        \varepsilon e^{i\frac{\theta}{2}} & \varepsilon e^{i\frac{\theta+\pi}{2}}
    \end{pmatrix}$.
\item[(iii)] $M=H_2(\mu_0)=
  \begin{pmatrix}
      & 1\\
     \mu_0 & \\
  \end{pmatrix}.$
  \item[(iv)] $M=\varepsilon_0\,e^{i\nicefrac{\theta_0}{2}}\Delta_1 \oplus \varepsilon_1\,e^{i\nicefrac{\theta_1}{2}}\Delta_1=\left(\begin{array}{c|c}
\varepsilon_0\,e^{i\nicefrac{\theta_0}{2}} &  \\
\hline
 & \varepsilon_1\,e^{i\nicefrac{\theta_1}{2}} \\
\end{array}\right).$
\end{itemize}
\end{lemma}

\begin{proof}
It suffices to analyze which are the possible different eigenvalues of $C(M)$, their modulus and the dimension of their corresponding eigenspaces.  As notation, eigenvalues will be denoted as $\mu_i$ and $W_i$ will refer to the corresponding eigenspaces.  In the case of values of modulus 1, we will write $\mu_j = e^{i\theta_j}$.

Let us start with the case where there is only one eigenvalue $\mu_0$ of multiplicity 2, $p_{C(M)}(z)=(z-\mu_0)^2$.  By Lemma \ref{eigenvalue}, $|\mu_0|=1$.
Remark~\ref{hermHS} gives directly case (i) and allows us to consider the case 
when $dim(W_0)=1$. Then 
        \begin{equation*}\label{J2J1}
            C(M) \sim J_2(\mu_0)=\begin{pmatrix}
                \mu_0&1\\
                &\mu_0
            \end{pmatrix},
        \end{equation*}
and therefore $M=\varepsilon e^{i\nicefrac{\theta_0}{2}}\Delta_2$, obtaining case (ii) (see Lemma~\ref{simil-cosquare} and the algorithm for the regular case).
 
 The other possibility if the case when $C(M)$ has two distinct eigenvalues $\mu_0, \mu_1$ each of them of multiplicity one. Then $p_{C(M)}(z)=(z-\mu_0)(z-\mu_1)$ and $C(M) \sim Diag(\mu_0, \mu_1)$. However, regarding the modulus of the eigenvalues, two situations can happen: on the one hand, if one of them, let say $\mu_0$, has modulus not equal to 1, then $\nicefrac{1}{\overline{\mu_0}}$ must be also an eigenvalue, so $Sp(C(M))=\{\mu_0, \nicefrac{1}{\overline{\mu_0}}\}$ and 
 $M=H_2(\mu_0)$ (case (iii)).
 
 The missing situation is when both of the eigenvalues have modulus 1, which will provide case (iv).  Now, $M=\varepsilon_0\,e^{i\nicefrac{\theta_0}{2}}\Delta_1 \oplus \varepsilon_1\,e^{i\nicefrac{\theta_1}{2}}\Delta_1$.
\end{proof}

Using the previous lemma and rescaling the last element of the basis of $\mathfrak g^{1,0}$, we obtain the remaining classification in dimension 6: 

\begin{theorem}\label{classHS}
    Let $(\mathfrak{g}, J_M)$ be a $6$-dimensional NLA with a $1$-abelian complex structure where $M$ is a regular non-Hermitian matrix.  Then, there exists a basis $\{\sigma^i\}_{i=1}^3$ of $\mathfrak{g}^{1,0}$ such that the complex structure equations are one of the following:
\begin{itemize}
    \item[(a)]$J^3_{r1}:\quad d\sigma^1=d\sigma^2=0$, $d\sigma^3=\sigma^{1\bar 2}+\sigma^{2\bar 1}+i\sigma^{2\bar 2}$;\\[-5pt]
    \item[(b)] $J^3_{r2}:\quad d\sigma^1=d\sigma^2=0$, $d\sigma^3=\sigma^{1\bar 2}+\mu\,\sigma^{2\bar 1}$ for $\mu \in \mathbb{R}$, $1<\mu<+\infty$;\\[-5pt]
    \item[(c)] $J^3_{r3}:\quad d\sigma^1=d\sigma^2=0$, $d\sigma^3=\sigma^{1\bar 1}+\lambda\,\sigma^{2\bar 2}$ for $\lambda=e^{i\theta}$, $\theta \in (0, \pi).$
\end{itemize}
\end{theorem}

\begin{proof}
Consider a generic complex structure equations in dimension 6 given by \eqref{ouracs}, where the matrix $M$ and the basis $\{\tau^i\}_{i=1}^3$ is determined by Lemma~\ref{lema: block-dim6}.  For case (i) in the previous lemma, we are reduced to the Hermitian case, so it is already studied in Theorem~\ref{6-Hermitian}. To obtain equations (a), consider 
$$\tau^i=\sigma^i, \,\, i=1,2,\quad \tau^3=\varepsilon\, e^{i\nicefrac{\theta}{2}} \sigma^3.$$ 
applied to case (ii). Case (b) follows directly from case (iii) considering the change $$\tau^1=e^{i\frac{Arg(\mu_0)}{2}}\sigma^1, \quad \tau^2=\sigma^2, \quad \tau^3=e^{i\frac{Arg(\mu_0)}{2}}\sigma^3$$ where $\mu = |\mu_0|$.  Finally, in order to obtain equations in case (c), note that considering the following change of basis starting from (iv) in Lemma~\ref{lema: block-dim6},
$$\tau^1=\eta^1, \quad \tau^2=\eta^2, \quad \tau^3=\varepsilon_0\,e^{i\nicefrac{\theta_0}{2}}\,\eta^3,$$ we obtain (c) with,  $\lambda=\dfrac{\varepsilon_1e^{i\nicefrac{\theta_1}{2}}}{\varepsilon_0e^{i\nicefrac{\theta_0}{2}}}=\pm e^{i\frac{\theta_1-\theta_0}{2}}$. Nevertheless, if we apply the change $$\eta^1=i\sigma^2, \quad \eta^2=\sigma^1, \quad \eta^3=\lambda\sigma^3,$$
we arrive again to (c), but in this case with coefficient $\nicefrac{1}{\lambda}$. So, in (c) we can consider that $\arg \lambda\in (0,\pi)$.
\end{proof}

\bigskip

Next, we identify the real 6-dimensional NLAs that admit 1-abelian complex structure $J_M$ according to the classification obtained in Theorems~\ref{6-singular}, ~\ref{6-Hermitian} and~\ref{classHS}. 

\begin{theorem}\label{main-theorem-6}
Let $(\mathfrak{g}, J_M)$ be a $6$-dimensional NLA with a $1$-abelian complex structure.  Then $(\mathfrak{g}, J_M)$ is isomorphic to:
\begin{itemize}
\item Singular of maximal type case: $(\mathfrak{h}_5, J^3_{s1})$;\\[-5pt]
\item Regular Hermitian case: $(\mathfrak{h}_3, J_{3,2,0})$; $(\mathfrak{h}_3, J_{3,1,0})$; \\[-5pt]
\item Regular and non-Hermitian case: $(\mathfrak{h}_4, J^3_{r1})$; $(\mathfrak{h}_5, J^3_{r2})$; $(\mathfrak{h}_2, J^3_{r3})$,
\end{itemize}
where we have used Salamon's notation~\cite{Salamon} for real nilpotent Lie algebras of dimension~6. 
\end{theorem}

\begin{proof}
It suffices to define a real basis $\{e^1,\ldots, e^6\}$ taking appropriate real and imaginary parts of the complex basis $\{\sigma^1,\sigma^2, \sigma^3\}$ for each complex equations~$J_M$.  The results appear in Table~\ref{Tabla-dim6-reducida}, where we have added the singular Hermitian case for completeness.
\end{proof}

\renewcommand{\arraystretch}{1.5}
\begin{table}[h!]
\begin{tabular}{|c|c|l|c|}
\hline
$J_M$&*congruence blocks&Real basis&NLA\\
\hline
$0$&& &Abelian\\ 
\hline
&$\Delta_1 \oplus J_1(0)$&$\sigma^1=e^1+ie^2,\quad \sigma^2=e^3+ie^4$&\\
$J_{3,1,1}$&Hermitian singular&$\sigma^3=e^5-2ie^6$&$\frak{h}_8$ \\ 
%&&$\omega^3=e^5-2ie^6$&\\
\hline \hline
&&$\sigma^1=e^1+ie^2,\quad \sigma^2=e^3-ie^4$& \\
$J^3_{s1}$&$J_2(0)$&$\sigma^3=e^5+ie^6$&$\frak{h}_5$ \\
%&&$\omega^3=e^5+ie^6$& \\
\hline \hline
&&$\sigma^1=e^1\pm ie^2,\quad \sigma^2=e^3+ie^4$&\\
$J_{3,2,0},\, J_{3,1,0}$&Hermitian regular&$\sigma^3=e^5 \mp 2ie^6$&$\frak{h}_3$ \\
%&matrix&$\omega^3=e^5 \mp 2ie^6$&\\
\hline \hline
%&&$\omega^1 = e^1 + ie^2,\quad \omega^2 = e^3 + ie^4,$& \\ 
$J^3_{r3}$&$e^{i\theta_1}\Delta_1 \oplus e^{i\theta_2}\Delta_1$&$\sigma^1 = e^1 + ie^2,\quad \sigma^2 = e^3 + ie^4,$&$\frak{h}_2$ \\ 
$\lambda = a+ib$&&$\sigma^3 = 2b e^6 - 2(e^5 + a e^6) i$& \\ 
\hline
%&&$\omega^1 = e^1 + ie^2$& \\ 
$J^3_{r2}$&$H_2(\mu)$&$\sigma^1 = e^1 + ie^2,\quad \sigma^2 = e^3 - ie^4$&$\frak{h}_5$ \\ 
$|\mu|>1$&&$\sigma^3 = \frac{e^5}{1-\mu} + i\frac{e^6}{1+\mu}$& \\ 
\hline
%&&$\omega^1 = e^3 - ie^4$,& \\ 
$J^3_{r1}$&$e^{i\theta}\Delta_2$&$\sigma^1 = e^3 - ie^4,\quad \sigma^2 = e^1 + ie^2$,&$\frak{h}_4$ \\ 
&&$\sigma^3 = 2e^5 + 2i\,e^6$&\\ 
\hline
\end{tabular}
\caption{Nilpotent Lie algebras of dimension~6 admitting 1-abelian complex structures.}\label{Tabla-dim6-reducida}
\end{table}

%The identification appears in Table \ref{Tabla-dim6}, where we have given explicitly the relation between the complex and the real basis for each Lie algebra.  We have used Salamon's notation for real nilpotent Lie algebras of dimension 6 \cite{Salamon}, where:
%\begin{eqnarray*}
%\mathfrak{h_2} = (0,0,0,0,12, 34),&\quad &\mathfrak{h_5} = (0,0,0,0,13+42, 14+23),\\
%\mathfrak{h_3} = (0,0,0,0,0, 12+34),&\quad& \mathfrak{h_8} = (0,0,0,0,0, 12),\\
%\mathfrak{h_4} = (0,0,0,0,12,14+23),&\quad&
%\end{eqnarray*}
%and $(0,0,0,0,13+42, 14+23)$ means that there exists a real basis of $1$-forms, $\{e^1,\ldots, e^6\}$ such that $de^1 = de^2=de^3=de^4 = 0$, $de^5 = e^1\wedge e^3 + e^4\wedge e^2$ and $de^6 = e^1\wedge e^4 + e^2\wedge e^3$.

%%%%%%%%%%%%%%%%%%%%%%%%%%%%%%%%%%%%%%%%%%

%For case $(iv)$, take the following real basis:
%$$\omega^1 = e^1 + ie^2,\quad \omega^2 = e^3 + ie^4,\quad \omega^3 = 2b e^6 - 2(e^5 + a e^6) i,\quad \text{where \, } \mu = a+ib.$$

%For case $(v)$, take the following real basis:
%$$\omega^1 = e^1 + ie^2,\quad \omega^2 = e^3 + ie^4,\quad \omega^3 = \frac{e^5}{1-\mu} + i\frac{e^6}{1+\mu}.$$

%For case $(vi)$, take the following real basis:
%$$\omega^1 = e^3 - ie^4,\quad \omega^2 = e^1 + ie^2,\quad \omega^3 = 2e^5 + 2i\,e^6.$$

It is worth noticing that in each orbit determined in Theorem~\ref{classHS} only one Lie algebra appears, in sharp contrast to other classification approaches (see for instance \cite{ABD, COUV}). In the following example, we illustrate how our method can be used to classify a specific 1-abelian complex structure in dimension six and identify its underlying Lie algebra.

\begin{example} Consider the complex 1-abelian structure $J_M$ determined by the equations:
\begin{equation*}\label{t}
    d\omega^1 = d\omega^2 = 0,\quad d\omega^3 = \omega^{1\bar1} -\frac{i}{t}\omega^{1\bar2} -\frac{i}{t}\omega^{2\bar1}  -\omega^{2\bar2},\quad t\in(0,1].
\end{equation*}
According to \cite[Theorem 3.5]{ABD}, (see also \cite[page 13]{COUV}), it corresponds to a family of abelian complex structures of $\mathfrak{h}_5$.  Let us recover this result by applying our method for classification.  First of all, the associated matrix to the structure $J_M$ is simply: $$M=\begin{pmatrix}
    1 & -\nicefrac{i}{t}\\
    -\nicefrac{i}{t} & -1
\end{pmatrix}.$$
Observe that $|M| = 0$ if and only if $t=1$.  So, we distinguish two cases:
if $t=1$, then $M$ is a singular matrix of rank $(M) = 1$.  Applying Steps 1 and 2 of the regularization algorithm, $m_1=1$ and there exists an elementary matrix $P=\begin{pmatrix}
    1 & 0\\
    i & 1
\end{pmatrix}$ such that  $PMP^*=\begin{pmatrix}
    1 & -2i\\
    0 & 0
\end{pmatrix}$.  In particular $m_2 = 1$ and according to Theorem~\ref{6-singular}, the complex structure $J_M$ is equivalent to $J_{s1}^3$ and the underlying Lie algebra is $\mathfrak h_5$ (see Table~\ref{Tabla-dim6-reducida}).

On the other hand, if $t \in (0,1)$, then, $M$ is a regular matrix that is never Hermitian, so it is necessary to analyze its *cosquare. Direct calculations give: $$(M^*)^{-1}=\frac{t^2}{1-t^2}\begin{pmatrix}
        -1 & \nicefrac{-i}{t}\\
        \nicefrac{-i}{t} & 1
    \end{pmatrix},\quad \text{and }\quad C(M)=(M^*)^{-1}M=\begin{pmatrix}
        -\frac{1+t^2}{1-t^2} & \frac{2it}{1-t^2}\\
        -\frac{2it}{1-t^2} & -\frac{1+t^2}{1-t^2}
    \end{pmatrix}.$$

The next step is to diagonalize $C(M)$ by similarity. The characteristic polynomial is $P_{C(M)}(\lambda)=\lambda^2+2\frac{1+t^2}{1-t^2}\lambda+1$ and its (real) solutions are:
$$\lambda_1 = \frac{-(1-t)^2}{1-t^2},\quad \lambda_2 = \frac{-(1+t)^2}{1-t^2}.$$
Observe that $|\lambda_2|>1$, since $t \in (0,1)$ so, we conclude that (see Lemma~\ref{lema: block-dim6}, (iii)) $$M \sim H_2(\lambda_2) = \begin{pmatrix}
    0 & 1\\
    \lambda_2& 0
\end{pmatrix}.$$
In particular, the complex structure determined by $M$ is equivalent to: 
$$d\sigma^1 = d\sigma^2 = 0,\quad d\sigma^3 = \sigma^{1\bar2} + \lambda_2\,\sigma^{2\bar1}.$$
According to Table \ref{Tabla-dim6-reducida}, the underlying Lie algebra is $\mathfrak{h}_5$.
\end{example}

\bigskip

We finish this section by comparing our classification with the one provided in~\cite{COUV}. We recall their main result:

\begin{theorem}\cite[Corollary 3.1, Table 1]{COUV} 
    Let $(\mathfrak{g}, J)$ be a $6$-dimensional NLA with a $1$-abelian complex structure. Then there is a basis $\{\omega^i\}_{i=1}^3$ for $\mathfrak{g}^{1,0}$ satisfying one (and only one) of the following equations \begin{enumerate}
        \item $d\omega^1=d\omega^2=d\omega^3=0$.\\[-5pt]
        \item $d\omega^1=d\omega^2=0$, $d\omega^3=\omega^{1\bar 1}$, and the Lie algebra is $\mathfrak h_8$.  \\[-5pt]
        \item $d\omega^1=d\omega^2=0$, $d\omega^3=\omega^{1\bar 1}\pm \omega^{2\bar 2}$, and the Lie algebra is $\mathfrak h_3$. \\[-5pt] 
        \item $d\omega^1=d\omega^2=0$, $d\omega^3=\omega^{1\bar 1}+D\omega^{2\bar 2}$, with $D \in \mathbb{C}, \,\mathfrak{Im}D=1$, and the Lie algebra is $\mathfrak h_2$.  \\[-5pt]
        \item $d\omega^1=d\omega^2=0$, $d\omega^3=\omega^{1\bar 1}+\omega^{1\bar 2}+D\omega^{2\bar 2}$, with $D \in [0,\frac14)$, and the Lie algebra is $\mathfrak h_5$.  \\[-5pt]
         \item $d\omega^1=d\omega^2=0$, $d\omega^3=\omega^{1\bar 1}+\omega^{1\bar 2}+\frac14\omega^{2\bar 2}$, and the Lie algebra is $\mathfrak h_4$.
    \end{enumerate}
\end{theorem}

It is straightforward to observe that case (1) is the abelian situation, case (2) is singular of not maximal type and case (3) is the Hermitian one.  Moreover, for cases (4), (5) and (6) in the previous theorem, the eigenvalues of the corresponding $C(M)$ are distinct of modulus 1 (for case (4)) distinct with modulus different from 1 (for case (5)) and a unique eigenvalue (for case (6)), which agrees with our previous results. 

\medskip

In order to properly compare  both approaches, let us consider a 1-abelian complex structure $J_M$ defined by a (non-Hermitian) matrix 
\begin{equation}\label{matrix-M-dim6}
M = \begin{pmatrix}
1&\lambda\\ 0&D
\end{pmatrix},
\end{equation}
as it is done in~\cite{COUV}.  Let us study the behavior of the spectral properties of $C(M)$ depending on the values of the pair $(\lambda, D)$.  Direct computations show that 
\begin{equation}\label{cosquare-dim6}
C(M)=\begin{pmatrix}
        1 & \lambda\\   
        -\dfrac{\lambda}{\bar D} & \dfrac{D-\lambda^2}{\bar D}
    \end{pmatrix}.
    \end{equation}
From the explicit form of $C(M)$, we can deduce the similarity class of the *cosquare. This is explained in the following proposition:

\begin{proposition}\label{sim}
Let $M$ be a non-singular matrix defined as in \eqref{matrix-M-dim6} with $D=x+iy \neq 0,\, \lambda\in\mathbb R$ and $C(M)$ its *cosquare \eqref{cosquare-dim6}. Then:
\begin{itemize}
\item[(i)] $C(M)$ has an eigenvalue of algebraic multiplicity 2 if and only if $y^2 = \lambda^2\left(\frac14 \lambda^2 - x\right])$.\\[-5pt]
\item[(ii)] $C(M)$ has two different eigenvalues $z_1$ and $z_2$ with $|z_i|\neq 1$ if and only if $y^2 < \lambda^2\left(\frac14 \lambda^2 - x\right)$.\\[-5pt]
\item[(iii)] $C(M)$ has two different eigenvalues $z_1$ and $z_2$ with $|z_i|=1$ if and only if $y^2 > \lambda^2\left(\frac14 \lambda^2 - x\right)$.
\end{itemize}

\end{proposition}

\begin{proof}
The characteristic polynomial $P_{C(M)}(z)$ of $C(M)$ is given by: 
\begin{equation}\label{pol-car}
P_{C(M)}(z)=z^2-\Biggl(1+\frac{D-\lambda^2}{\overline{D}}\Biggl)z+\frac{D}{\overline{D}},
\end{equation}
and its discriminant is $$\Delta=\Biggl(1+\frac{D-\lambda^2}{\overline{D}}\Biggl)^2-\frac{4D}{\overline{D}} =  \frac{4}{(\bar D)^2}\left[\lambda^2\left(\frac14 \lambda^2-x\right) - y^2\right],$$ where we have expressed $D:=x+iy$ with $x,y \in \mathbb{R}$.  
Let us denote 
\begin{equation*}\label{delta-tilde}
    \widetilde \Delta := \lambda^2\left(\frac14 \lambda^2-x\right) - y^2. 
\end{equation*} Observe that $\widetilde \Delta$ is a real number.

With these notations, the solutions of \eqref{pol-car} are given by:
\begin{equation*}\label{solutions-pol-car}
z = \frac{1}{2\bar D} \left[2x-\lambda^2 \pm 2\sqrt{\widetilde\Delta}\right].
\end{equation*}

Clearly, $\widetilde\Delta=0$ is equivalent of having an eigenvalue of algebraic multiplicity 2 and case (i) is completely finished.

Suppose now that $\widetilde\Delta>0$.  Then, $\sqrt{\widetilde\Delta} : =\alpha\in\mathbb R\setminus\{0\}$, and the two (different) eigenvalues are:
$$z_1 = \frac{2x-\lambda^2+2\alpha}{2\bar D},\quad z_2 =\frac{2x-\lambda^2-2\alpha}{2\bar D}.$$  Observe that $|z_1|=|z_2|$ if and only if $\alpha=0$ (not possible) or $\lambda^2 = 2x$.  In this last case, $\widetilde\Delta = -\frac14 \lambda^4-y^2\leq 0$, which is a contradiction.  Therefore, $|z_i|\neq 1$ and we are in case (ii).  Moreover, it can be proved that $z_2 = \nicefrac{1}{\bar z_1}$.

Finally, if $\widetilde\Delta<0$, we can express $\sqrt{\widetilde\Delta} : =i \alpha,$ where $\alpha\in\mathbb R\setminus\{0\}$.  Now, the two distinct eigenvalues are:
$$z_1 = \frac{2x-\lambda^2+2i\alpha}{2\bar D},\quad z_2 = \frac{2x-\lambda^2-2i\alpha}{2\bar D}.$$
A straightforward calculation gives $|z_1| = |z_2| = 1$, so we are in case (iii).

For the converse, suppose first that $P_{C(M)}(z)$ has two different roots, $z_1\neq z_2$.  According to Lemma~\ref{eigenvalue}, two cases can happen: (a) $z_2 = \frac{1}{\overline{z_1}}$ where $|z_i|\neq 1$; or (b) $|z_1| = |z_2| = 1$. 

Using Cardano-Viète formulas: $$\begin{cases}
        z_1+z_2=1+\frac{D-\lambda^2}{\overline{D}},\\
        z_1z_2=\frac{D}{\overline{D}}.
    \end{cases}$$

    Now, in case $(a)$, we have that:
    $$\left|1+\frac{D-\lambda^2}{\overline{D}}\right| = |z_1+z_2| = \left|z_1 + \frac{1}{\overline{z_1}}\right| = \left|\frac{|z_1|^2+1}{\overline{z_1}}\right| = \frac{|z_1|^2+1}{\left|{z_1}\right|} = |z_1| + \frac{1}{|z_1|}>2,$$ where the last inequality is strict since $|z_1|\neq 1$ and holds for any non-zero real number.  The condition $\left|1+\frac{D-\lambda^2}{\overline{D}}\right|>2$ is equivalent to say $y^2 < \lambda^2\left(\frac{1}{4}\lambda^2-x\right)$ (case $(ii)$).

    In case $(b)$, 
    $$\left|1+\frac{D-\lambda^2}{\overline{D}}\right| = |z_1+z_2| < |z_1|+|z_2|=2,$$ and the inequality is strict since $z_1\neq z_2$.  Observe that this condition is equivalent to $y^2 > \lambda^2\left(\frac{1}{4}\lambda^2-x\right)$ (case $(iii)$).
\end{proof}

The previous result combining with Lemma~\ref{lema: block-dim6} provides the link between the two classifications:

%gives us the number of orbits that we can have depending on the value of $D$ and $\lambda$. In other words, Proposition \ref{sim} gives us the number of possible *congruence classes that can be. Next step would be describe this *congruence classes by finding a representative element or each one. Is here where the arguments of section 3.1 make their contribution

\begin{theorem}
    Let $M$ be defined as in~\eqref{matrix-M-dim6} with $D=x+iy \neq 0$. According to the cases listed in Lemma~\ref{lema: block-dim6}:
    \begin{itemize}
    \item[(I)] If $y^2 = \lambda^2(\frac{1}{4}\lambda^2-x)$, then \begin{itemize}
        \item[(I.1)] If $\lambda=0$, then $M$ is a Hermitian matrix (type (i)).
        \item[(I.2)] If $\lambda \neq 0$ then $M$ is of type (ii). \\[-5pt]
    \end{itemize}
    \item[(II)] If $y^2 < \lambda^2(\frac{1}{4}\lambda^2-x)$, then $M$ is of type (iii).\\[-5pt]
    \item[(III)] If $y^2 > \lambda^2(\frac{1}{4}\lambda^2-x)$, then $M$ is of type (iv).
    \end{itemize}
\end{theorem}

%%%%%%%%%%%%%%%%%%%%%%%%%%%%%%%%%%%%%%%%%%%%%%%%%%%%%%%%%%%%%%%%%%%%%
\subsection{The 8-dimensional case}\label{sub:8}

In this section we classify the 1-abelian complex structures on nilpotent Lie algebras of dimension 8, using the same procedure as before.  We will study first the Hermitian case and then the non-Hermitian one depending if the matrix $M$ is regular or not. 

\subsubsection{The Hermitian case in dimension 8}
Consider a 1-abelian complex structure with a Hermitian matrix $M$.   According to Theorem~\ref{teo:Hermitian}, since $n=4$ is an even number and the rank of the matrix $M$ can be 1, 2 or 3, the complex structures $J_{4,k,r}$ (see equations~\eqref{eq:Hermitian}) and the underlying Lie algebras are the following:

    \begin{theorem}\label{Hermitian}
       Let $(\mathfrak{g}, J)$ be an $8$-dimensional NLA with $1$-abelian complex structure with $M$ a Hermitian matrix of rank $4-r$. Then there exists a basis $\{\sigma^i\}_{i=1}^4$ of $\mathfrak{g}^{1,0}$ such that the complex structure equations $J_{4,k,r}$ are given by:
       $$J_{4,3,0} :  \begin{cases}
            d\sigma^1=d\sigma^2=d\sigma^3=0\\
            d\sigma^4=\sigma^{1\bar 1}+\sigma^{2\bar 2}+\sigma^{3\bar 3},
        \end{cases} \quad J_{4,2,0} :\begin{cases}
            d\sigma^1=d\sigma^2=d\sigma^3=0\\
            d\sigma^4=\sigma^{1\bar 1}+\sigma^{2\bar 2}-\sigma^{3\bar 3},
        \end{cases}$$
        $$J_{4,2,1} :  \begin{cases}
            d\sigma^1=d\sigma^2=d\sigma^3=0\\
            d\sigma^4=\sigma^{1\bar 1}+\sigma^{2\bar 2},
        \end{cases} \quad J_{4,1,1} :\begin{cases}
            d\sigma^1=d\sigma^2=d\sigma^3=0\\
            d\sigma^4=\sigma^{1\bar 1}-\sigma^{2\bar 2},
        \end{cases}\quad J_{4,1,2} :\begin{cases}
            d\sigma^1=d\sigma^2=d\sigma^3=0\\
            d\sigma^4=\sigma^{1\bar 1}.
        \end{cases}$$
        %$$J_{4,1,2} :\begin{cases}
         %   d\sigma^1=d\sigma^2=d\sigma^3=0\\
        %    d\sigma^4=\sigma^{1\bar 1}.
        %\end{cases}$$
    \end{theorem}
    Moreover, the underlying Lie algebras are of Heisenberg type.  Concretely:
    $$J_{4,3,0},\, J_{4,2,0}\rightarrow \mathfrak h_{4,0},\quad J_{4,2,1},\, J_{4,1,1}\rightarrow \mathfrak h_{4,1},\quad J_{4,1,2}\rightarrow \mathfrak h_{4,2}.$$ 
 Recall that $\mathfrak h_{4,1}$ and $\mathfrak h_{4,2}$ are products of Lie algebras (see \eqref{eq:relation-heis}).  In fact, complex structures of maximal type are $J_{4,3,0}$ and $J_{4,2,0}$ (recall Remark~\ref{Hermitian-max}).

\subsubsection{The regular non-Hermitian case in dimension 8}

If $M$ is regular but not Hermitian, similar to the discussion in dimension 6, $S=0$ in Theorem~\ref{class} and the number of blocks $\Delta_j$ and $H_{2i}(\mu)$ is determined by the similarity type of its *cosquare, $C(M)$. In this case $C(M)\in GL(3,\mathbb C)$ and its Jordan canonical form is determined by its eigenvalues $\eta_0, \eta_1, \eta_2$ and the dimension of the associated eigenspaces ($W_i$). 

In the following result we establish the possibilities for the block decomposition of $M$ analyzing the spectral properties of $C(M)$: 

\begin{lemma}\label{lema: block-dim8}
Let $(\mathfrak{g}, J_M)$ be an 8-dimensional NLA with a 1-abelian complex structure where $M$ is a regular non-Hermitian matrix, given by \eqref{ouracs}.  Then, there exist a basis $\{\tau^i\}_{i=1}^4$ of $\mathfrak{g}^{1,0}$ for which matrix $M$ is one of the following, where $\varepsilon, \varepsilon_i = \pm 1$:
\begin{itemize}
\item[(i)] $M= c\widetilde M$, where $|c| = 1$ and $\widetilde M$ is Hermitian;
\item[(ii)] $M=\varepsilon_1e^{i\nicefrac{\theta}{2}}\Delta_2 \oplus \varepsilon_2e^{i\nicefrac{\theta}{2}}\Delta_1=\left(\begin{array}{@{}c|c@{}}
  \begin{matrix}
        & \varepsilon_1e^{i\frac{\theta}{2}}\\
       \varepsilon_1e^{i\frac{\theta}{2}} & \varepsilon_1e^{i\frac{\theta+\pi}{2}}
  \end{matrix}
  &  \\
\hline
   &
  \varepsilon_2e^{i\frac{\theta}{2}}
\end{array}\right)$;
\item[(iii)] $M=\varepsilon e^{i\nicefrac{\theta}{2}}\Delta_3=\begin{pmatrix}
        & & \varepsilon e^{i\frac{\theta}{2}}\\
        & \varepsilon e^{i\frac{\theta}{2}} & \varepsilon e^{i\frac{\theta+\pi}{2}}\\
        \varepsilon e^{i\frac{\theta}{2}} & \varepsilon e^{i\frac{\theta+\pi}{2}} &
    \end{pmatrix}$;
  \item[(iv)] $
    M=\varepsilon_0e^{i\nicefrac{\theta_0}{2}}\Delta_1\oplus\varepsilon_1e^{i\nicefrac{\theta_0}{2}}\Delta_1\oplus\varepsilon_2e^{i\nicefrac{\theta_1}{2}}\Delta_1  =\left(\begin{array}{c|c|c}
\varepsilon_0e^{i\nicefrac{\theta_0}{2}} &  &  \\
\hline
 & \varepsilon_1e^{i\nicefrac{\theta_0}{2}} & \\
\hline
 &  & \varepsilon_2e^{i\nicefrac{\theta_1}{2}}
\end{array}\right)$;
\item[(v)] $
        M=\varepsilon_1e^{i\nicefrac{\theta_0}{2}}\Delta_2 \oplus \varepsilon_2e^{i\nicefrac{\theta_1}{2}}\Delta_1=\left(\begin{array}{@{}c|c@{}}
  \begin{matrix}
        & \varepsilon_1e^{i\frac{\theta_0}{2}}\\
       \varepsilon_1e^{i\frac{\theta_0}{2}} & \varepsilon_1e^{i\frac{\theta_0+\pi}{2}}
  \end{matrix}
  &  \\
\hline
   &
  \varepsilon_2e^{i\frac{\theta_1}{2}}
\end{array}\right)$;
\item[(vi)] $
    M=H_2(\eta_0)\oplus \varepsilon e^{i\nicefrac{\theta_1}{2}}\Delta_1=\left(\begin{array}{@{}c|c@{}}
  \begin{matrix}
      & 1\\
     \eta_0 & \\
  \end{matrix}
  &  \\
\hline
   &
  \varepsilon e^{i\frac{\theta_1}{2}}
\end{array}\right)$;

\item[(vii)] $M=\varepsilon_0e^{i\nicefrac{\theta_0}{2}}\Delta_1 \oplus \varepsilon_1e^{i\nicefrac{\theta_1}{2}}\Delta_1 \oplus \varepsilon_2e^{i\nicefrac{\theta_2}{2}}\Delta_1 
        =\left(\begin{array}{c|c|c}
\varepsilon_0e^{i\nicefrac{\theta_0}{2}} &  &  \\
\hline
 & \varepsilon_1e^{i\nicefrac{\theta_1}{2}} & \\
\hline
 &  & \varepsilon_2e^{i\nicefrac{\theta_2}{2}}
\end{array}\right).$
\end{itemize}
\end{lemma}

\begin{proof}
In a similar way of the proof of Lemma~\ref{lema: block-dim6}, it suffices to analyze which are the possible different eigenvalues of $C(M)$, their modulus and the dimension of their corresponding eigenspace.  As notation, eigenvalues will be denoted as $\mu_0$, $W_i$ will refer to the corresponding eigenspaces.  In the case of values of modulus 1, we will write $\mu_j = e^{i\theta_j}$.

Let us start with the case where there is only one eigenvalue $\mu_0$ of multiplicity 3, $p_{C(M)}(z)=(z-\mu_0)^3$.  By Lemma \ref{eigenvalue}, $|\mu_0|=1$.
Remark~\ref{hermHS} gives directly case (i) and allows us to consider the case when $dim(W_0)\neq 3$.  If $dim(W_0)=2$, then 
        \begin{equation*}
            C(M) \sim J_2(\eta_0)\oplus J_1(\eta_0)=\left(\begin{array}{@{}c|c@{}}
  \begin{matrix}
      \eta_0 & 1\\
       & \eta_0
  \end{matrix}
  &  \\
\hline
   &
  \eta_0
\end{array}\right),
        \end{equation*}
and therefore, the block form of $M$ is 
$M=\varepsilon_1e^{i\nicefrac{\theta_0}{2}}\Delta_2 \oplus \varepsilon_2e^{i\nicefrac{\theta_0}{2}}\Delta_1,$ obtaining case (ii).  In contrast, if  $dim(W_0)=1$, then
        \begin{equation*}
            C(M) \sim J_3(\eta_0)=\begin{pmatrix}
                \eta_0&1&\\
                &\eta_0&1\\
                &&\eta_0
            \end{pmatrix},
        \end{equation*}
and the block decomposition for $M$ is simply $M=\varepsilon e^{i\nicefrac{\theta_0}{2}}\Delta_3$, which gives case (iii).

 The second option related with eigenvalues of $C(M)$ is the existence of one eigenvalue $\eta_0$ of algebraic multiplicity~2, and another eigenvalue $\eta_1$ of multiplicity~1, that is, $p_{C(M)}(z)=(z-\eta_0)^2(z-\eta_1)$. As in the previous case, by Lemma \ref{eigenvalue}, both $\eta_0$ and $\eta_1$ must satisfy that $|\eta_0|=|\eta_1|=1$, (indeed, if $|\eta_1|>1$, then $\nicefrac{1}{\overline{\eta_1}}$ would be also an eigenvalue with the same multiplicity~2, but this is impossible due to the size of $M$). We can write $\eta_0=e^{i\theta_0}$ and $\eta_1=e^{i\theta_1}$ with $\theta_1\neq\theta_0+2k\pi$.  We have two possibilities for the Jordan canonical form of $C(M)$: 
firstly, if $dim(W_0)=2 \text{ and }dim(W_1)=1$, then $C(M)$ is diagonalizable and 
$C(M) \sim Diag(\eta_0, \eta_0, \eta_1)$ and $
    M=\varepsilon_0e^{i\nicefrac{\theta_0}{2}}\Delta_1\oplus\varepsilon_1e^{i\nicefrac{\theta_0}{2}}\Delta_1\oplus\varepsilon_2e^{i\nicefrac{\theta_1}{2}}\Delta_1,$ which correspond with case (iv).  Secondly, $dim(W_0)=1 \text{ and }dim(W_1)=1$, then \begin{equation*}
            C(M) \sim J_2(\eta_0)\oplus J_1(\eta_1)=\left(\begin{array}{@{}c|c@{}}
  \begin{matrix}
      \eta_0 & 1\\
       & \eta_0
  \end{matrix}
  &  \\
\hline
   &
  \eta_1
\end{array}\right).
        \end{equation*}
        In this case, the matrix $M$ decomposes into two blocks as $
        M=\varepsilon_1e^{i\nicefrac{\theta_0}{2}}\Delta_2 \oplus \varepsilon_2e^{i\nicefrac{\theta_1}{2}}\Delta_1$, providing case (v).

Finally, the last option is when $C(M)$ has three different eigenvalues $\eta_0, \, \eta_1, \, \eta_2$ each of them of algebraic multiplicity 1, then $p_{C(M)}(z)=(z-\eta_0)(z-\eta_1)(z-\eta_2)$ and $C(M) \sim Diag(\eta_0, \eta_1, \eta_2)$. Regarding the modulus of the eigenvalues, two situations can happen: 
on the one hand, if one of them, let say $\eta_0$, has modulus not equal to 1, then $\nicefrac{1}{\overline{\eta_0}}$ must be also an eigenvalue, so $Sp(C(M))=\{\eta_0, \nicefrac{1}{\overline{\eta_0}}, \eta_1\}$ and $\eta_1$ must have modulus 1. If we express $\eta_1=e^{i\theta_1}$, then $   
    M=H_2(\eta_0)\oplus \varepsilon e^{i\nicefrac{\theta_1}{2}}\Delta_1$, corresponding to case (vi).  On the other hand, if all the eigenvalues have modulus 1, then $M = \varepsilon_0e^{i\nicefrac{\theta_0}{2}}\Delta_1 \oplus \varepsilon_1e^{i\nicefrac{\theta_1}{2}}\Delta_1 \oplus \varepsilon_2e^{i\nicefrac{\theta_2}{2}}\Delta_1$, as in case (vii).
\end{proof}

\medskip

The previous discussion allows us to consider 1-abelian complex structures $J_M$ with regular non-Hermitian matrices $M$ of forms (ii),\ldots, (vii) in Lemma~\ref{lema: block-dim8}.  However, in order to classify complex structures up to equivalence, we have to take into account the rescaling condition.  The main result of the section is the following:
\begin{theorem}\label{Th-dim8-rg3}
    Let $(\mathfrak{g}, J_M)$ be an 8-dimensional NLA with a 1-abelian complex structure defined by a regular non-Hermitian matrix $M$.  Then there exists a basis $\{\sigma^i\}_{i=1}^4$ of $\mathfrak{g}^{1,0}$ such that the complex structure equations are $d\sigma^1=d\sigma^2=d\sigma^3=0$ and one of the following:
\begin{eqnarray*}J^4_{r1} &:& 
            d\sigma^4=\sigma^{1\bar 3}+\sigma^{2\bar 2}+\sigma^{3\bar 1}+i(\sigma^{3\bar 2}+\sigma^{2\bar 3});\\[1pt]
J^4_{r2} &:&
            d\sigma^4=\sigma^{1\bar 2}+\sigma^{2\bar 1}+i\sigma^{2\bar 2}\pm\sigma^{3\bar 3};\\[1pt]
J^4_{r3} &:&
            d\sigma^4=\sigma^{1\bar 2}+\sigma^{2\bar 1}+i\sigma^{2\bar 2}+z\sigma^{3\bar 3},\quad |z|=1,\, z\neq \pm 1; \\[1pt]
J^4_{r4+}&:&
            d\sigma^4=\sigma^{1\bar 1}+ \sigma^{2\bar 2}+z\sigma^{3\bar 3},\quad |z|=1,\, z\neq \pm 1;\\[1pt]
J^4_{r4-}&:&
            d\sigma^4=\sigma^{1\bar 1}- \sigma^{2\bar 2}+z\sigma^{3\bar 3},\quad |z|=1,\,\, z=e^{i\theta},\, \theta\in(0,\pi);\\[1pt]
J^4_{r5}&:&
            d\sigma^4=\sigma^{1\bar 2}+z_0\sigma^{2\bar 1}+\sigma^{3\bar 3},\quad |z_0|>1;\\[1pt]
J^4_{r6}& : &
            d\sigma^4=\sigma^{1\bar 1}+z\sigma^{2\bar 2}+z_1\sigma^{3\bar 3},\quad |z|=1 \text{ and } z\neq \pm 1;\,\, |z_1| = 1,\,  \\ &&\phantom{d\sigma^4=\sigma^{1\bar 1}+z\sigma^{2\bar 2}+z_1\sigma^{3\bar 3},} \text{ and } z_1 = e^{i\theta}, \theta\in(0, \pi).\text{ Moreover } \, z\neq \pm z_1. 
 \end{eqnarray*}       
        %where $z, z_0, z_1 \in \mathbb{C}$ with $|z|=1$ and $z\neq \pm 1$. In $(V)$, $|z_0|>1$ and, in $(VI)$, $|z_0|=|z_1|=1$ and $z_0 \neq \pm z_1,\, z_0, z_1\notin \mathbb R$. 
\end{theorem}
\begin{proof}
Consider a 1-abelian complex structures $J_M$ with a regular non-Hermitian matrix $M$ of the form (ii),\ldots, (vii).  Let us study case by case:
\begin{itemize}
    \item If $M$ is of form (ii), just consider the new basis $\{\sigma^1,\ldots, \sigma^4\}$ given by
    \begin{equation}\label{eq:cambio-reescalamiento}\omega^i=\sigma^i, \quad \omega^4=\varepsilon_1e^{i\nicefrac{\theta}{2}}\sigma^4,
    \end{equation}
    for $i=1,2,3$, to obtain the expression $J^4_{r2}$.
    \item  For matrices of form (iii), the same basis \eqref{eq:cambio-reescalamiento} provides $J^4_{r1}$.
    \item In the case (v), we can also consider basis \eqref{eq:cambio-reescalamiento} to obtain $J^4_{r3}$ with $z=\dfrac{\varepsilon_2e^{i\nicefrac{\theta_1}{2}}}{\varepsilon_1e^{i\nicefrac{\theta_0}{2}}}=\pm e^{i\frac{\theta_1-\theta_0}{2}}$. Note that since $\eta_0 \neq \eta_1$, then, $z\neq \pm 1$.
     \item For matrices of form (iv), the basis (writing $\tau^j$ instead of $\sigma^j$) \eqref{eq:cambio-reescalamiento} gives the following equation for $d\tau^4$, which will be identified with the pair $(\varepsilon, z)$: 
\begin{equation}\label{IV}
    d\tau^4=\tau^{1\bar 1}+\varepsilon \tau^{2\bar 2}+z\tau^{3\bar 3},
\end{equation} with $\varepsilon \in \{1, -1\}$ and $z=\dfrac{\varepsilon_2e^{i\nicefrac{\theta_1}{2}}}{\varepsilon_0e^{i\nicefrac{\theta_0}{2}}}=\pm e^{i\frac{\theta_1-\theta_0}{2}}$. Note that as $\eta_0 \neq \eta_1$, then $\theta_1\neq\theta_0+2k\pi$, so $z\neq \pm 1$. If $\varepsilon=1$, then equation \eqref{IV} provides directly case $J^4_{r4+}$. If $\varepsilon=-1$, the change $$\tau^1=\sigma^2, \, \tau^2=\sigma^1, \, \tau^3=\sigma^3, \, \tau^4=-\sigma^4$$ gives an equivalence between the complex structures represented by $(-1, z)$ and $(-1, -z)$ 
So we can take $z$ with $\mathfrak{Im}(z) > 0$ obtaining case $J^4_{r4-}$.
\item If $M$ is of form (vi), take the new basis $\{\sigma^1,\ldots, \sigma^4\}$ given by: 
$$\omega^1=\varepsilon\, e^{i\nicefrac{\theta_1}{2}}\,\sigma^1,\quad \omega^2=\sigma^2,\quad \omega^3=\sigma^3,\quad \omega^4=\varepsilon\, e^{i\nicefrac{\theta_1}{2}}\,\sigma^4,$$
getting equation $J^4_{r5}$ with $z_0=\varepsilon\,\eta_0\,e^{i\theta_1}$. Note that $|z_0|=|\eta_0|>1$.
\item Finally, for matrices of form (vii), in terms of the basis (writing $\tau^j$ instead of $\sigma^j$) \eqref{eq:cambio-reescalamiento} we get 
$$
    d\tau^4=\tau^{1\bar 1}+z_0\tau^{2\bar 2}+z_1\tau^{3\bar 3}
$$ with $z_0=\dfrac{\varepsilon_1e^{i\nicefrac{\theta_1}{2}}}{\varepsilon_0e^{i\nicefrac{\theta_0}{2}}}=\pm e^{i\frac{\theta_1-\theta_0}{2}}$ and $z_1=\dfrac{\varepsilon_1e^{i\nicefrac{\theta_2}{2}}}{\varepsilon_0e^{i\nicefrac{\theta_0}{2}}}=\pm e^{i\frac{\theta_2-\theta_0}{2}}$. 

As before, we identify the previous complex structure with pair $(z_0, z_1)$. The change 
$$\tau^1=z_1\,\sigma^3, \quad \tau^2=\sigma^2, \quad \tau^3=\sigma^1, \, \tau^4=z_1\,\sigma^4$$ gives us an equivalence between $(z_0, z_1)$ and $\left(\dfrac{z_0}{z_1}, \overline{z_1}\right)$. In the latter case, note that $\dfrac{z_0}{z_1}$ has modulus 1. Thanks to this change, we can take $z_1$ with $\mathfrak{Im}(z_1) > 0$ obtaining case $J^4_{r6}$, where $z=\dfrac{z_0}{z_1}$.
\end{itemize}
\end{proof}

To conclude this section, we determine the underlying Lie algebras of each family of complex structures listed in Theorem~\ref{Th-dim8-rg3}.

\begin{theorem}\label{teo_algebras-regular}
Let $(\mathfrak{g}, J_M)$ be an $8$-dimensional NLA with a $1$-abelian complex structure defined by a regular non-Hermitian matrix $M$.  Then the Lie algebra $\mathfrak{g}$ is isomorphic to one in the following list where in all cases $\{e^1,\ldots, e^8\}$ is a basis of 1-forms where $e^j,\, j=1,\ldots, 6$ are closed:
\begin{itemize}
%\item[(i)] $\mathfrak{g}^8_{r1}: \quad de^7 = e^{36} - e^{45},\quad
%de^8 = e^{16} - e^{25} +e^{34}$.\\[-5pt]
%\item[(ii)] $\mathfrak{g}^8_{r2}: \quad de^7 = e^{34},\quad
%de^8 = e^{14} - e^{23} +e^{56}$.\\[-5pt]
%\item[(iii)] $\mathfrak{g}^8_{r3}: \quad de^7 = e^{34}+e^{56},\quad
%de^8 = e^{14} - e^{23}$.\\[-5pt]
%\item[(iv)] $\mathfrak{g}^8_{r4}: \quad de^7 = e^{12},\quad
%de^8 = e^{34} +e^{56}$.\\[-5pt]
%\item[(v)] $\mathfrak{g}^8_{r5}: \quad de^7 = e^{13} + e^{24} + e^{56},\quad
%de^8 = e^{14} - e^{23}.$\\[-5pt] %+ \lambda e^{56},$ where $\lambda\in[0,1]$.
%\item[(vi)] $\mathfrak{g}^8_{r6}: \quad de^7 = e^{12}+e^{34},\quad
%de^8 = e^{34} +e^{56}$.
\item[(i)] $\mathfrak{g}^8_{r1}: \quad de^7 = e^{13} + e^{24} + e^{56},\quad
de^8 = e^{36} - e^{45}$.\\[-5pt]
\item[(ii)] $\mathfrak{g}^8_{r2}: \quad de^7 = e^{13} + e^{24} + e^{56},\quad
de^8 = e^{34}$.\\[-5pt]
\item[(iii)] $\mathfrak{g}^8_{r3}: \quad de^7 = e^{13}+e^{24},\quad
de^8 = e^{34} + e^{56}$.\\[-5pt]
\item[(iv)] $\mathfrak{g}^8_{r4}: \quad de^7 = e^{12},\quad
de^8 = e^{34} +e^{56}$.\\[-5pt]
\item[(v)] $\mathfrak{g}^8_{r5}: \quad de^7 = e^{13} + e^{24} + e^{56},\quad
de^8 = e^{14} - e^{23}.$\\[-5pt] %+ \lambda e^{56},$ where $\lambda\in[0,1]$.
\item[(vi)] $\mathfrak{g}^8_{r6}: \quad de^7 = e^{12}+e^{34},\quad
de^8 = e^{34} +e^{56}$.
\end{itemize}
\end{theorem}

\begin{proof}
To obtain the result it suffices to decompose each $\sigma^j,\, j=1,2,3,4,$ appearing in Theorem~\ref{Th-dim8-rg3} in suitable real and imaginary parts.
\begin{itemize}
\item[(i)] From $J^4_{r1}$, consider $$\sigma^1 = e^1+i\,e^2,\quad \sigma^2 = e^6-i\,e^5,\quad\sigma^3 = -e^4+i\,e^3,\quad\sigma^4 = -2e^8-2i\,e^7,$$ to obtain the equations of $\mathfrak g^8_{r1}$.
\item[(ii)] If we consider $J^4_{r2}$, just take: $$\sigma^1 = e^1+i\,e^2,\quad \sigma^2 = -e^4+i\,e^3,\quad\sigma^3 = \varepsilon\, e^5+i\,e^6,\quad\sigma^4 = 2e^8-2i\,e^7,$$ where $\varepsilon = \pm 1$.
\item[(iii)] In $J^4_{r3}$ appears a complex number $z$, where $|z| =1$ but $z\neq \pm 1$.  Let us express $z = a + i\,b$.  Since $z\notin \mathbb R$, then $b\neq 0$.  However, $a$ could be zero or not.  

If $a=0$, it suffices to consider the following basis $\{e^1,\ldots, e^8\}$, where:
$$\sigma^1 = e^1+i\,e^2,\quad \sigma^2 = -e^4+i\,e^3,\quad\sigma^3 = \frac1b\, e^5+i\,e^6,\quad\sigma^4 = 2e^8-2i\,e^7.$$

On the other hand, if $a\neq 0$, consider:
$$\sigma^1 = \frac{a}{b}\left[(e^1 - e^4)+i\,e^2\right],\quad \sigma^2 = -e^4+i\,e^3,\quad\sigma^3 = \frac1b\, e^5+i\,e^6,\quad\sigma^4 = 2e^8-\frac{2a}{b}i\,(e^7 + e^8).$$
\item[(iv)] Observe that we can join the expressions of $J^4_{r4+}$ and $J^4_{r4-}$ simply as: $d\sigma^1=d\sigma^2=d\sigma^3=0$ and $d\sigma^4=\sigma^{1\bar 1}+\varepsilon\, \sigma^{2\bar 2}+z\sigma^{3\bar 3}$, 
where $\varepsilon = \pm 1$, and $|z| =1$ but $z\neq \pm 1$.  Let us express $z = a + i\,b$.  Since $z\notin \mathbb R$, then $b\neq 0$. In order to obtain the result, just take:
$$\sigma^1 = -e^5+i\,e^6,\quad \sigma^2 = -\varepsilon\,e^3+i\,e^4,\quad\sigma^3 = e^1+i\,e^2,\quad\sigma^4 = 2b\,e^7+2i\,(e^8 - a\,e^7).$$
\item[(v)] Starting from $J^4_{r5}$, consider first
$$\sigma^1 = \nu^1 + i \nu^2,\quad \sigma^2 = \nu^3 + i \nu^4,\quad \sigma^3 = \nu^5 + i \nu^6,\quad \sigma^4 = \nu^7 + i \nu^8,$$ and let us denote $z_0 = a+ib$.
With these notations, it is immediate to see that:
$$d\nu^1=\cdots = d\nu^6=0,$$ $$d\nu^7 = (1-a)(\nu^{13} + \nu^{24}) - b\, (\nu^{23} - \nu^{14}),$$ $$d\nu^8 = -b\,(\nu^{13} + \nu^{24}) + (1+a) (\nu^{23} - \nu^{14}) - 2\,\nu^{56}.$$

Now, we can clear $\nu^{13} + \nu^{24}$ and $\nu^{23} - \nu^{14}$ in the equations above.  Concretely:
$$\nu^{13} + \nu^{24} = \frac{1}{1-|z_0|^2} \left[ d((1+a)\nu^7 + b\,\nu^8) + 2\,b\,\nu^{56}\right],$$
$$\nu^{23} - \nu^{14} = \frac{1}{1-|z_0|^2} \left[ d(b\,\nu^7 + (1-a)\nu^8) + 2\,(1-a)\,\nu^{56}\right].$$
Observe that $(1+a)\nu^7 + b\,\nu^8$ and $b\,\nu^7 + (1-a)\nu^8$ are linearly independent one-forms (since $|z_0|\neq 1$).  Moreover, $b$ and $(1-a)$ cannot be zero simultaneously.  Without loss of generality, we can suppose that  $b\neq 0$ and define the new basis of real one-forms:
$$f^i = \nu^i,\quad i=1,\ldots, 5,$$
$$f^6 = -\frac{2\,b}{1-|z_0|^2} \nu^6,\quad f^7 = \frac{(1+a) \nu^7 + b\,\nu^8}{1-|z_0|^2},\quad f^8 = -\frac{b\,\nu^7 + (1-a)\nu^8}{1-|z_0|^2}.$$
In terms of this basis, we obtain the following structure equations:
$$df^i = 0,\, i=1,\ldots, 6,\quad df^7 = f^{13} + f^{24} + f^{26},\quad df^8 = f^{14} - f^{23} +\lambda\, f^{26},$$

where $\lambda = \dfrac{a-1}{b}\in\mathbb R$.  The last step is to show that $\lambda$ can be suppose to be zero.  For that, just consider $\theta:= \frac12\arctan \lambda$ and the new basis $\{e^i\}_{i=1}^8$ given by:
$$
e^1 = f^1\,\cos\theta - f^2\sin\theta,\qquad
e^2 = f^1\,\sin\theta + f^2\cos\theta,$$ $$
e^3 = f^3\,\cos\theta + f^4\sin\theta,\qquad
e^4 = -f^3\,\sin\theta + f^4\cos\theta,
$$
$$
e^5=f^5,\quad e^6=\sqrt{1+\lambda^2}\,f^6,\quad e^7=\frac{f^7 + \lambda\, f^8}{\sqrt{1+\lambda^2}},\quad e^8=\frac{-\lambda\, f^7 + f^8}{\sqrt{1+\lambda^2}}.
$$

%Finally, let us see how to simplify $\lambda$ whenever $\lambda\neq 0$: first,  there exists an isomorphism between $\mathfrak g_{\lambda}$ and $\mathfrak g_{1/\lambda}$, just taking the new basis:
%$$\widetilde e^1 = e^2,\quad \widetilde e^2 = e^1,\quad\widetilde e^3 = -e^3,\quad\widetilde e^6 = \lambda e^6,\quad\widetilde e^7 = e^8,\quad\widetilde e^8 = e^7.$$
%Secondly, there exists an isomorphism between $\mathfrak g_{\lambda}$ and $\mathfrak g_{-\lambda}$.  For that, take new basis as:
%$$\widetilde e^2 = -e^2,\quad \widetilde e^3 = -e^3,\quad\widetilde e^5 = -e^5,\quad\widetilde e^7 = -e^7.$$
\item[(vi)] Finally, for $J^4_{r6}$, let us express $z=a_1+i\,b_1$ and $z_1 = a_2+i\,b_2$.  Observe that the above conditions imply that $b_1b_2\neq 0$ and $\frac{a_1}{b_1} \neq \frac{a_2}{b_2}$.  For this last inequality, take into account that $\frac{a_1}{b_1} = \frac{a_2}{b_2}$ if and only if $\tan\,\alpha = \tan \,\theta$ and therefore $z=\pm z_1$, which is not possible. 

       To obtain the real Lie algebra, it is enough to consider the following real basis $\{e^1,\,\ldots, e^8\}$, where:
\begin{eqnarray*}
\sigma^1 = \left(\frac{a_1}{b_1} - \frac{a_2}{b_2} \right) e^1 + i\, e^2,&\qquad &
\sigma^2 = \frac{1}{b_1}e^3 + i\, e^4,\\[5pt]
\sigma^3 = \frac{1}{b_2}e^5 + i\, e^6,&\qquad &
\sigma^4 = 2\,e^8 -2 i\, \left[\left(\frac{a_1}{b_1} - \frac{a_2}{b_2} \right) e^7 + \frac{a_2}{b_2} e^8\right].
\end{eqnarray*}
\end{itemize}
\end{proof}

%%%%%%%%%%%%%%%%%%%%%%%%%%%%%%%%%%%%%%%%%%%%%%%%%%%%%%%%%%%%%%%%%%%%%%%%
\subsubsection{The singular non-Hermitian case of maximal type in dimension $8$}
In this section we will consider 1-abelian complex structures $J_M$ whose matrices $M$ are non-Hermitian and whose rank is not maximal, i.e. the matrix $M$ will have rank $r<3$. Recall that we are interested in singular matrices $M$ of maximal type since they provide new complex structures.  According to Lemma~\ref{lemma-rango-maximal}, $3>r\geq 2$, so the only possibility is $r=2$.  Moreover, by definition of maximal type, we are forced to have $m_1=m_2=1$ and therefore, $M_{(1)}$ appearing in $M^{(iv)}$ (see~\eqref{reducedform}) has dimension 1 (see Remark~\ref{rmk:singular-m}).  The following result gives us the possible matrices $M$:

\begin{lemma}\label{lema-block8-sing}
Let $(\mathfrak{g}, J_M)$ be an $8$-dimensional NLA with a $1$-abelian complex structure where $M$ is a singular non-Hermitian matrix of maximal type.  Then, there exists a basis $\{\tau^i\}_{i=1}^4$ of $\mathfrak{g}^{1,0}$ for which matrix $M$ is one of the following:
\begin{itemize}
\item[(i)] $M=\begin{pmatrix}
    A & 0 & 0\\
    0 & 0 & 1\\
    0 & 0 & 0
\end{pmatrix},$ where $A\in \mathbb C^*$. 
\item[(ii)] $M=\begin{pmatrix}
    0 & 1 & 0\\
    0 & 0 & 1\\
    0 & 0 & 0
\end{pmatrix}.$
\end{itemize}
\end{lemma}

\begin{proof}
As we have explained before, $m_1=m_2=1$ and the matrix~\eqref{reducedform} is given by $$M^{(iv)}=\begin{pmatrix}
    A & B & 0\\
    0 & 0 & 1\\
    0 & 0 & 0
\end{pmatrix},$$ where $A, B \in M_1(\mathbb{C})$.  Now, $A$ can be a regular or a singular matrix, i.e. $A\neq 0$ or $A=0$.  If $A\neq 0$, then we can take $B=0$ (see \eqref{B=0}) obtaining case (i).  On the other hand, if $A=0$, by \eqref{A=0}, $B$ can be normalize to be 1, getting case (ii). Note that this case corresponds to completely singular matrices.
\end{proof}

As in the previous cases, in order to classify complex structures up to equivalence, we have to take into account the rescaling condition.  

\begin{theorem}\label{Th-dim8-rg2}
    Let $(\mathfrak{g}, J_M)$ be an $8$-dimensional NLA with a $1$-abelian complex structure defined by a singular non-Hermitian matrix $M$ of maximal type.  Then there exists a basis $\{\sigma^i\}_{i=1}^4$ of $\mathfrak{g}^{1,0}$ such that the complex structure equations are $d\sigma^1=d\sigma^2=d\sigma^3=0$ and one of the following:
\begin{itemize}
\item[(i)] $J^4_{s1}: \quad 
            d\sigma^4=\sigma^{1\bar 2}+\sigma^{2\bar 3}$ ;\\[-5pt]
\item[(ii)] $J^4_{s2}: \quad  d\sigma^4=\sigma^{1\bar 1}+\sigma^{2\bar 3}$.
\end{itemize}
\end{theorem}  

\begin{proof}
Equations for $J^4_{s1}$ follow directly from matrix $M$ in Lemma~\ref{lema-block8-sing}, case (ii).  For $J^4_{s2}$, starting from $M$ in Lemma~\ref{lema-block8-sing}, case (i), we obtain $d\omega^4 = A\omega^{1\bar1} + \omega^{2\bar3}.$  Now, the new basis $\{\sigma^i\}_{i=1}^4$ given by
$$\sigma^1 = \omega^1,\quad \sigma^2 = \frac{1}{A}\omega^2,\quad \sigma^3 = \omega^3,\quad \sigma^4 = \frac{1}{A}\omega^4,$$ provides the simplified complex equations. 
\end{proof}

Regarding real Lie algebras, the following result holds:

\begin{theorem}
Let $(\mathfrak{g}, J_M)$ be an $8$-dimensional NLA with a $1$-abelian complex structure defined by a singular non-Hermitian matrix $M$ of maximal type.  Then the Lie algebra $\mathfrak{g}$ is isomorphic to one in the following list where in all cases $\{e^1,\ldots, e^8\}$ is a basis of 1-forms where $e^j,\, j=1,\ldots, 6$ are closed:
\begin{itemize}
\item[(i)] $\mathfrak{g}^8_{s1}: \quad de^7 = e^{13} + e^{24},\quad 
de^8 = e^{36} + e^{45}$;\\[-5pt]
\item[(ii)] $\mathfrak{g}^8_{r5}$.
\end{itemize}
\end{theorem}

\begin{proof}
In case $J^4_{s1}$, we consider:
$$\sigma^1 = \left(\frac{e^1+e^5}{2}\right)+i\,\left(\frac{e^2-e^6}{2}\right),\qquad \sigma^2 = e^3+i\,e^4,$$
$$\sigma^3 = \left(\frac{e^5-e^1}{2}\right)-i\,\left(\frac{e^2+e^6}{2}\right),\qquad\sigma^4 = e^7+i\,e^8.$$

For $J^4_{s2}$, it suffices to consider the real basis $\{e^1\ldots, e^8\}$ defined as:
$$\sigma^1 = \frac{e^5}{2}+i\,e^6,\quad \sigma^2 = e^3+i\,e^4,\quad\sigma^3 =  e^2-i\,e^1,\quad\sigma^4 = e^8-i\,e^7.$$
In this case we obtain a nilpotent Lie algebra that has already appeared in the regular case (see Theorem~\ref{teo_algebras-regular}).
\end{proof}

All the results in dimension 8 are summarized in Table~\ref{Tabla-dim8}.  It is worth noticing that all the nilpotent Lie algebras have first Betti number equal to 6 except the Heisenberg one (see Propositions~\ref{prop:topo} and~\ref{prop:espectral}).  Moreover, this classification shows that all the complex structures $J_M$ derived from a particular block decomposition of *congruence of the matrix $M$ live in the same Lie algebra.

\begin{table}[h!]
\begin{tabular}{|l|c|l|}
\hline
Complex structure:&*congruence&Nilpotent Lie Algebra\\
$d\sigma^1 = d\sigma^2 = d\sigma^3 =0,\quad d\sigma^4 = \ldots$&blocks&\\
\hline
$J_{4,3,0}: \,\sigma^{1\bar 1}+\sigma^{2\bar 2} +\sigma^{3\bar 3}$&\multirow{2}{*}{Hermitian regular}&\multirow{2}{*}{$\mathfrak h_{4,0}: \begin{cases}
de^7 = 0,\\
de^8 = e^{12} + e^{34} +e^{56}.
\end{cases}$}\\
\cline{1-1}
$J_{4,2,0}:\, \sigma^{1\bar 1}+\sigma^{2\bar 2} -\sigma^{3\bar 3}$&&\\
\hline
\hline
$J^4_{r1}:\, \sigma^{1\bar 3}+\sigma^{2\bar 2}+\sigma^{3\bar 1}+i(\sigma^{3\bar 2}+\sigma^{2\bar 3})$&$e^{i\theta}\Delta_3$&$\mathfrak g^8_{r1}:\begin{cases}
de^7 = e^{13} + e^{24} + e^{56},\\
de^8 = e^{36} -e^{45}.
\end{cases}$\\
\hline
$J^4_{r2}:\,\sigma^{1\bar 2}+\sigma^{2\bar 1}+i\sigma^{2\bar 2}+\varepsilon\,\sigma^{3\bar 3},\, \varepsilon=\pm 1$&$e^{i\theta}\Delta_2 \oplus e^{i\theta}\Delta_1$&$\mathfrak g^8_{r2}: \begin{cases}
de^7 = e^{13} + e^{24} +e^{56},\\
de^8 = e^{34}.
\end{cases}$\\ 
\hline
$J^4_{r3}:\,\sigma^{1\bar 2}+\sigma^{2\bar 1}+i\sigma^{2\bar 2}+z\sigma^{3\bar 3},\quad |z|=1,$&$e^{i\theta_1}\Delta_2 \oplus e^{i\theta_2}\Delta_1$&$\mathfrak g^8_{r3}:\begin{cases}
de^7 = e^{13}+e^{24},\\
de^8 = e^{34} + e^{56}.
\end{cases}$\\
$z\neq \pm 1,\quad z = a+\,i\,b,\quad b\neq 0$&&  \\
%&$z = a+\,i\,b,\quad b\neq 0$&&\\
\hline
$J^4_{r4+}:\,\sigma^{1\bar 1}+ \sigma^{2\bar 2}+z\sigma^{3\bar 3},\quad |z|=1,\, z\neq \pm 1$&\multirow{3}{*}{$e^{i\theta_1}\Delta_1\oplus e^{i\theta_1}\Delta_1\oplus e^{i\theta_2}\Delta_1$}&\multirow{3}{*}{$\mathfrak g^8_{r4}:\begin{cases}
de^7 = e^{12},\\
de^8 = e^{34} +e^{56}.
\end{cases}$} \\
%&$|z|=1,\, z\neq \pm 1$&&\\
\cline{1-1}
$J^4_{r4-}:\,\sigma^{1\bar 1}- \sigma^{2\bar 2}+z\sigma^{3\bar 3},\quad |z|=1,$&
%$\Delta_1\oplus\Delta_1\oplus\Delta_1$
&\\
%$\mathfrak g^8_{r4}:\begin{cases}
%de^7 = e^{12},\\
%de^8 = e^{34} +e^{56}.
%\end{cases}$ \\
$\mathfrak{Im}(z)>0$&&\\
\hline
$J^4_{r5}:\,\sigma^{1\bar 2}+z_0\sigma^{2\bar 1}+\sigma^{3\bar 3},\quad |z_0|>1$&$H_2(\mu)\oplus e^{i\theta}\Delta_1$&$\mathfrak g^8_{r5}:\begin{cases}
de^7 = e^{13} + e^{24} + e^{56},\\
de^8 = e^{14} - e^{23},
\end{cases}$ \\ 
%&&$\text{ where } \lambda\in[0,1]$ \\ 
\hline
$J^4_{r6}:\,\sigma^{1\bar 1}+z\sigma^{2\bar 2}+z_1\sigma^{3\bar 3},\quad |z|=|z_1| =1 $&$e^{i\theta_1}\Delta_1\oplus e^{i\theta_2}\Delta_1\oplus e^{i\theta_3}\Delta_1$&$\mathfrak g^8_{r6}:\begin{cases}
de^7 = e^{12}+e^{34},\\
de^8 = e^{34} +e^{56}.
\end{cases}$ \\ 
$z\neq \pm 1;\,\, z_1 = e^{i\theta}, \theta\in(0, \pi), \, z\neq \pm z_1.$&& \\ 
%$z_1 = e^{i\theta}, \theta\in(0, \pi) \, z\neq \pm z_1.$&&\\
\hline
\hline
$J^4_{s1}: \sigma^{1\bar 2}+\sigma^{2\bar 3}$&Singular&$\mathfrak g^8_{s1}:\begin{cases}
de^7 = e^{13} + e^{24},\\
de^8 = e^{36} +e^{45}.
\end{cases}$\\
\hline
$J^4_{s2}: \sigma^{1\bar 1}+\sigma^{2\bar 3}$&Singular&$\mathfrak g^8_{r5}:\begin{cases}
de^7 = e^{13} + e^{24} + e^{56},\\
de^8 = e^{14} - e^{23}.
\end{cases}$\\
\hline
\end{tabular}
\caption{Nilpotent Lie algebras of dimension~8 admitting 1-abelian complex structures of maximal type.}\label{Tabla-dim8}
\end{table}

%%%%%%%%%%%%%%%%%%%%%%%%%%%%%%%%%%%%%%

\subsection{The general case}\label{sub:general}
Having established the classification by equivalence of 1-abelian complex structures on NLAs of dimensions 4, 6, and 8, we now turn our attention to the case of arbitrary dimension $2n$. The purpose of this section is to investigate and enumerate the total number of families of 1-abelian complex structures of maximal type occurring in a fixed dimension $2n$, giving a lower bound for it.  This number will depend on the number and type of possible blocks under *congruence, as given in Theorem~\ref{class}. However, our results studying the Hermitian case and the particular case in dimension~8 show that the number of blocks is less than or equal to the number of non-equivalent families of complex structures.

Let us define: 
\begin{eqnarray*}Ab_1(n)&:=&\#\{\text{families  of 1-abelian complex structures of maximal type}\\ &&\text{  for a $2n$-dimensional NLA}\},\quad n\geq 2.\\
B(n-1)&:=&\#\{\text{decompositions as direct sum of blocks } \Delta_i, \, H_{2i}(\mu), J_i(0)\\ &&\text{  for maximal type } M \in M_{n-1}(\mathbb{C})\},\quad n\geq 2.
\end{eqnarray*}

As we have explained, $Ab_1(n)\geq B(n-1)$.

\medskip

Using the results of Sections~\ref{sub:4}, ~\ref{sub:6} and ~\ref{sub:8}, we can compute explicitly some values of $Ab_1(n)$ and $B(n-1)$.

\medskip

For $n=2$:  Clearly, $Ab_1(2)= B(1) = 1$ and the complex structure is $J_{2,1,0}$.

\medskip

For $n=3$: The results appear in Table~\ref{Tabla-dim6-reducida}.  Concretely, $Ab_1(3)=6$ and $B(2) = 5$.  Moreover, we have $2$ families of 1-abelian complex structures $J_M$ for $M$ Hermitian ($J_{3,2,0}$ and $J_{3,1,0}$) but both of them have the same configuration of *congruence blocks; $3$ families for $M$ regular and non-Hermitian ($J_{r1}^3, J_{r2}^3$ and  $J_{r3}^3$), and only 1 for $M$ singular of maximal type ($J^3_{s1}$).

\medskip

For $n=4$: The results appear in Table~\ref{Tabla-dim8}.  We have that, $Ab_1(4)=11$ and $B(3) = 9$.  Observe that we have two Hermitian structures corresponding to the same block configuration and also the complex structures $J^4_{r4+}, J^4_{r4-}$ have the same type of blocks.

\medskip

Since we have completely studied the Hermitian case, using Remark~\ref{Hermitian-max} together with Theorem \ref{teo:Hermitian} we conclude that $J_M$ with $M$ Hermitian of maximal type provide 1 configuration for blocks but $\lfloor \frac{n+1}{2}\rfloor$ families of non-equivalent complex structures.  Therefore, we obtain that:
\begin{equation}\label{AB-B}
Ab_1(n)\geq \lfloor \frac{n+1}{2}\rfloor + (B(n-1)-1).
\end{equation}
%where
%\begin{eqnarray*}
%B^*(n-1)&:=&\#\{\text{decompositions as direct sum of blocks } \Delta_i, \, H_{2i}(\mu), J_i(0)\\ &&\text{  for non-Hermitian maximal type } M \in M_{n-1}(\mathbb{C})\},\quad n\geq 2.
%\end{eqnarray*}
%Observe that $B^*(n-1) = B(n-1)-1$.

\subsubsection{The regular non-Hermitian case}\label{sect:general-regular}
Our present goal is to determine the number $B(n-1)$ for regular matrices of dimension $n-1$. Observe that Lemma~\ref{lema: block-dim6} and Lemma~\ref{lema: block-dim8} provide this number for $n=3$ and $n=4$ respectively. As can be seen, this problem is inherently combinatorics-driven, since it reduces to counting the viable combinations of blocks $H_{2k}(\mu)$ and $\Delta_k$.

As a matter of notation, let us use at this point $n$ for the dimension of the matrix.
 %We can ask how many possible block decompositions (up to rearrangement) there exist for regular matrices of a given dimension $n$.
 Let us define:
\begin{eqnarray*}a(n)&:=&\#\text{ block decompositions of a regular matrix $M$ of order }n.
\end{eqnarray*}

This number depends on the spectral properties of $C(M)$, in particular, depends on the modulus of the eigenvalues.  So, let us define two intermediate numbers:
\begin{eqnarray*}b(n)&:=&\#\text{ block decompositions of a regular matrix $M$ of order $n$} \\ &&\text{where all the eigenvalues of $C(M)$ are of modulus 1}.
\end{eqnarray*}
\begin{eqnarray*}c(n)&:=&\#\text{ block decompositions of a regular matrix $M$ of order $n$}\\ &&\text{ where there exist eigenvalues of $C(M)$ which are not of modulus 1.}
\end{eqnarray*}

Clearly, $a(n) = b(n) + c(n)$. Let us calculate first $b(n)$ and then we will relate $c(n)$ and $b(m)$ for some $m<n$

\begin{lemma}\label{lemma13}
Let $\{\mu_k\}_{k=1}^r$ be the eigenvalues of $C(M)$.  Suppose that all of them are of modulus~1 and denote by $m_k$ their algebraic multiplicities.  Then:
\begin{equation}\label{eq:bn}
b(n) = \sum_{m_1+\cdots+m_r = n} p(m_1)\cdots p(m_r),
\end{equation}
where $p(m_i)$ is number of the possible decompositions of $m_i$ as sum of natural numbers, that is, the number of partitions of $m_i$.
\end{lemma}

\begin{proof}
It is clear that $m_1+\cdots+m_r = n$ and $m_k\geq 1,\, \forall k$.  Now, if $W_k$ is the eigenspace associated to $\mu_k$, then its dimension $t_k$ can take values $t_k=1, 2,\ldots, m_k$. Moreover, $t_k$ determines also the number of Jordan blocks associated to the eigenvalue $\mu_k$. In other words, for a fixed value $1 \leq t_k \leq m_k$ there exist exactly $t_k$ Jordan blocks $J_{s_1}(\mu_k), \dots, J_{s_{t_k}}(\mu_k)$ whose sizes satisfy $$s_1 + \dots + s_{t_k} = m_k, \quad s_i \ge 1,$$and distinct choices of $(s_1, \dots, s_{t_k})$ (up to reordering) lead to non-similar Jordan matrices. Consequently, the number of Jordan forms compatible with $t_k$ blocks is precisely the number of partitions of $m_k$ into $t_k$ parts, and the total number of possible Jordan forms for $\mu_k$, with $1\leq t_k \leq m_k$, is$$\sum_{t_k=1}^{m_k} \#\{\text{partitions of } m_k \text{ into } t_k \text{ parts}\} = p(m_k).$$
So, each possible tuple $(m_1,\ldots, m_r)$ gives $p(m_1)\cdots p(m_r)$ options and the result follows.
\end{proof}

If $C(M)$ admit eigenvalues whose modulus is distinct of 1, these eigenvalues form pairs $(\mu, \bar \mu^{-1})$ and each pair generates a hyperbolic block, whose dimension depends on the multiplicity of $\mu$.  %In the next result we denote the number of hyperbolic blocks that can we have for each $n$ by $k$ and, the number of partitions of $k$ will be denoted by $r_k$, where we have counted how many different pairs (with multiplicity) can we have of each $n$:

\begin{lemma}\label{lemma14}
Suppose that $C(M)$ have pairs of eigenvalues with modulus different than $1$.  Then:
\begin{equation}\label{eq:cn}
c(n)=\sum_{k=1}^{\lfloor n/2\rfloor} b(k)\,b(n-2k).
\end{equation}
As notation, since $c(2) = 1$ and $b(1) = 1$, we will establish $b(0) = 1$. Moreover, we declare $c(0) = 0$ for completeness.
\end{lemma}

\begin{proof}
The key point is to analyze the number of pairs $(\mu_i,\bar \mu_i^{-1})$ counted by multiplicity that can appear inside a matrix of size $n$ and determine the left space for eigenvalues of modulus 1. Let $\mu_1,\dots,\mu_p$ be the representatives of the pairs $(\mu_i,\bar \mu_i^{-1})$ with $|\mu_i|>1$, and let $s_i$ be the algebraic multiplicity of $\mu_i$.  If we define $k:=s_1+\dots+s_p$, then $k \leq \lfloor n/2\rfloor$ and the remaining space, of dimension $n-2k$, is available for eigenvalues of modulus 1, which accounts for the factor $b(n-2k)$ in~\eqref{eq:cn}.  The factor $b(k)$ appears in a similar way but considering now the Jordan canonical form for the eigenvalues with modulus greater than $1$, taking into account that by Lemma \ref{simil-cosquare}, each Jordan block $J_s(\mu_i)$ corresponds, via the construction of the algorithm in Section \ref{section3.1}, to a hyperbolic block $H_{2s}(\mu_i)$. Finally, we must sum the product $b(k)b(n-2k)$ over each value of $k=1,\dots,\lfloor n/2\rfloor$.

%By Lemma \ref{simil-cosquare}, each Jordan block $J_s(\mu_i)$ corresponds, via the construction of the algorithm in Section \ref{section3.1}, to a hyperbolic block $H_{2s}(\mu_i)$. Fixing the algebraic multiplicity $s_i$ of $\mu_i$, by the same argument used in the proof of Lemma \ref{lemma13}, the number of non-similar Jordan forms for the eigenvalue $\mu_i$ is $p(s_i)$. Therefore, the number of configurations associated to a fixed tuple $(s_1,\dots,s_p)$ with $s_1+\dots+s_p=k$ is $p(s_1)\dots p(s_p)$, and the total number of configurations possible for the joint multiplicity $k$, summing over all partitions of $k$ into $p$ distinct representatives (with $p$ variable), is $$\sum_{s_1+\dots+s_p=k}p(s_1)\dots p(s_k)=b(k).$$
\end{proof}

Using the results of the previous lemmas, we can state the main result of this section:

\begin{proposition}
The number of block decompositions of $M\in GL(n,\mathbb C)$ is given by:
\begin{equation}\label{eq:an}
a(n) = b(n) + \sum_{k=1}^{\lfloor n/2\rfloor} b(k)\,b(n-2k),\quad n\geq 1;
\end{equation}
and $a(0) = 1$, where $b(n) = \displaystyle\sum_{m_1+\cdots+m_r = n} p(m_1)\cdots p(m_r)$ and $p(m)$ is the number of partitions of $m$.
\end{proposition}

\begin{example}\label{ex:numero-bloques}
Let us compute $a(n)$ for $n\leq 4$:
\begin{itemize}
\item $n=1$: The only possibility is an eigenvalue of modulus $1$ and algebraic multiplicity $1$, so the only tuple for algebraic multiplicities is $(m_1) = (1)$.  Now, $p(1) = 1$ and $b(1) = 1$.  Therefore, $a(1) = 1$.
\item $n=2$: We can have two options: two eigenvalues of modulus $1$ (this option gives us $b(2)$) or a pair with modulus different from $1$ (from which we will compute $c(2)$). Clearly, $c(2) = b(1)b(0)=1$.  For $b(2)$ we need first to determine the possible tuples $(m_1,m_2)$ such that $m_1+m_2 = 2$.  We obtain $(m_1) = 2$ and $(m_1,m_2) = (1,1)$.  Now, using~\eqref{eq:bn}:
$$b(2) = p(2) + p(1)p(1) = 2 + 1 = 3.$$
Finally: $a(2) = b(2) + c(2) = 3+1 = 4$, which agrees with Lemma~\ref{lema: block-dim6}.
\item $n=3$:  In this case, starting with $b(3)$, the possible tuples for the multiplicity of eigenvalues of modulus $1$ are:
$$(m_1,m_2, m_3) = (1,1,1),\, (1,2), \,(3) \Longrightarrow b(3) = p(1)^3 + p(1)p(2) + p(3) = 1 + 2 + 3 = 6.$$
What $c(n)$ concerns, observe that only can exist one pair $(\mu, \frac{1}{\overline{\mu}})$ and its multiplicity is $1$.  In particular, $k=1$, so: $$c(3) = b(1)b(3-2) = b(1)b(1) = 1.$$ 
Now: $a(3) = b(3) + c(3) = 6+1 = 7$, as stated in Lemma~\ref{lema: block-dim8}.
\item $n=4$.  First, $b(4)$ is given by:
$$(m_1, m_2, m_3, m_4) = (1,1,1,1),\,(1,1,2),\, (1,3),\, (2,2),\, (4)$$ $$\Longrightarrow b(4) = p(1)^4+p(1)^2p(2)+p(1)p(3)+p(2)^2+p(4)=1+2+3+4+5 = 15.$$
%On the other hand, we can count directly the possibilities for pairs $(\mu, \frac{1}{\overline{\mu}})$:  In fact, we can have $3$ options: only 1 pair with multiplicity $1$ (and space left for $2$ eigenvalues of modulus $1$), 1 pair with multiplicity $2$ (no more space) o $2$ pairs with multiplicity $1$ (no more space). Using the formula:
On the other hand, we can have up to two pairs $(\mu, \frac{1}{\overline{\mu}})$, so $k=1$ or $k=2$ and therefore:
$$c(4) = b(1)b(2) + b(2)b(0) = 3 + 3 = 6.$$
Finally: $a(4) = b(4) + c(4) = 21.$ That means that the analogous Lemma~\ref{lema: block-dim8} for dimension~$10$ would have $21$ items.
\end{itemize}
\end{example}

\subsubsection{The singular non-Hermitian case of maximal type}

In this section, we revisit the regularization algorithm (Section~\ref{sec:reg-alg}) in order to analyze how many possibilities there are for block decompositions, in the sense of the previous section.  We will consider a singular matrix $M$ of size $n$ and rank $r$ such that $J_M$ is of maximal type. According to Remark~\ref{rmk:singular-m}, the matrix $M_{(1)}$ has size $2r-n$, using that $m_1=m_2 = n-r$.  Let us denote $M_{(k)}$ the corresponding matrix obtained after finishing the $k-$th iteration of the algorithm.  Note that the associated $m_j$ that appear in this iteration are denoted by $m_{2k-1}, m_{2k}$.

Let us make several useful considerations about the regularization algorithm:
\begin{itemize}
\item In each iteration of the algorithm, the obtained matrix $M_{(k)}$ reduces its dimension respect to $M_{(k-1)}$ exactly in $m_{2k-1}+m_{2k}$.  In fact,
\begin{eqnarray}\label{dk}
d_1&:=&\text{dim }M_{(1)} = n-(m_1+m_2) = 2r-n, \nonumber \\  
d_k&:=&\text{dim }M_{(k)} = \displaystyle n-\sum_{j=1}^{2k}m_j = 2r-n - \sum_{j=3}^{2k}m_j,\quad k\geq 2,
\end{eqnarray}
where we have used that $m_1=m_2 = n-r$.
\item The algorithm can finish in different ways:
\begin{itemize}
    \item[(i)] $d_k=0$ for some $k$.  Then $M_{(k)}$ does not exist and $M$ is completely singular.  %we are in case 1 of Step 4 (see \eqref{m1+m2=n}). 
    \item[(ii)] $d_k\neq 0$ for all $k$ and there exists $j=2k$ or $j=2k+1$ such that $m_j=0$. Then $M_{(k)}$ is a regular matrix and $M = R\oplus S$.
\end{itemize} 

\item The *congruence invariants $m_i$ satisfy $m_j\geq m_k$ if $j<k$, so the longest algorithm is that for which $m_1=m_2 = n-r$ and $m_i=1$ for $i\geq 3$.  If $M_{(\tau)}$ is the reduced *congruence form of~$M$, then $2\tau= 2r-n+2$ or $2\tau - 1= 2r-n+2$.  In particular, $2\tau \leq  2r-n+2$. %where we have used that $m_1+\cdots+m_{\tau} = n$.
%\item The maximum possible number of iterations $\tau$ of the algorithm satisfies that $2\tau = 2r-n+2$ or $2\tau-1 = 2r-n+2$, depending on the parity of $n$.  In this case, $m_3 = \ldots = m_{2r-n+2} = 1$. %and the last $M'$ does not exist.
\end{itemize}

For our purpose, we define: 
\begin{equation}\label{tupla-m}
\vec{m} = (m_3,\ldots, m_{2r-n+2}),
\end{equation}
where:
\begin{equation}\label{conditions-singular}
\displaystyle \sum_{k=3}^{2r-n+2} m_k\leq 2r-n;\quad m_3\leq n-r;\quad m_i\geq m_j,\ \forall i\leq j;\quad m_i\geq 0.
\end{equation}
Any tuple $\vec{m}$ satisfying conditions~\eqref{conditions-singular} provides a new block decomposition for the singular part~$S$. Moreover, taking into account also the regular part $R$ in Theorem~\ref{class}, the number of new global block decompositions that a tuple $\vec{m}$ can provide is only $1$ if $M$ is completely singular, i.e, $d_k=0$ for some $k$ or, $a(d_{\tau})$ if $d_k\neq 0$ for all $k$, where $\tau$ refers to the final iteration of the algorithm and $d_k$ is given by \eqref{dk}. %where $M$, after the regularization algorithm, gives rise to a matrix $M_{(\tau)}$ of dimension $d_{\tau}$ given by~\eqref{dk}.  

\medskip

Let us denote by 
$\mathcal{D}$ the set of all admissible $\vec{m}$, i.e, $$\mathcal D = \{\vec{m}\in \mathbb R^{2r-n}\,|\, \vec{m} \text{ satisfies }\eqref{conditions-singular}\}.$$

Now, we can state the results:

\begin{lemma}\label{th4.13}
Let $(n,r)$ be a pair of natural numbers where $n-1\geq r\geq \lfloor \frac{n+1}{2}\rfloor$.  Then, the number of block decompositions associated to a singular matrix of maximal type of rank $r$, $M\in M_n(\mathbb C)$, is given by:
\begin{equation}\label{eq:singular-rango}
s(n,r) = \displaystyle \sum_{\vec{m} \in \mathcal{D}} a(d(\vec{m})),
\end{equation}
%$$s(n,r) = m_0(n,r) + \displaystyle \sum_{j\geq 1}  m_{2j-1}(n,r)+\displaystyle \sum_{j\geq 1}  a(2r-n-\sum_{k=3}^{2j} m_k) \, m_{2j}(n,r),$$ 
where $a(n)$ is given by~\eqref{eq:an}. %and we declare that $a(0) = 1$. 
\end{lemma}

\begin{proposition}
    Let $M\in M_n(\mathbb C)$ be a singular matrix of maximal type of rank $r$.  Then, the number of block decompositions is given by:
    \begin{equation}\label{eq:singular}
s(n) = \sum_{r=\lfloor \frac{n+1}{2}\rfloor}^{n-1} s(n,r),
    \end{equation}
    where $s(n,r)$ is defined in~\eqref{eq:singular-rango}.
\end{proposition}

Let us see some examples:

\begin{example}
Let us check that $s(3,2) = 2$, as we have obtained previously in Lemma~\ref{lema-block8-sing}. Let us suppose that $n=3$.  Then, the only possible rank for having a maximal type structure is $r=2$ (see Lemma~\ref{lemma-rango-maximal}). Now, $m_1 = m_2 = n-r = 1$. %and $M_{(1)}$ has dimension $d_1 = 2r-n=1$. 
Since $2r-n+2 = 3$, the longest tuple $\vec{m}$ reduces to $(m_3)$ and $m_3$ can only take the values 1 or 0. So, $\mathcal D = \{(1,1,1),\, (1,1,0)\}$.  Note that $\tau =2$, that is to say, the algorithm finishes in the second iteration.
\begin{itemize}
\item If $m_3 = 1$, observe that $d_2 = d(\vec m) = 0$, and the singular part $(m_1, m_2, m_3) = (1,1,1)$ provides $a(0) = 1$ new block decomposition, that corresponds to case (ii) in Lemma~\ref{lema-block8-sing}.  
\item On the other hand, if $m_3 = 0$, then $d_2 = d(\vec m) = 1$, which means that $M_{(\tau)} = M_{(2)}$ is a regular matrix of dimension $1$.  In this case, the singular part $(m_1, m_2) = (1,1)$ gives $a(1) = 1$ new block decomposition, as stated in case (i) in Lemma~\ref{lema-block8-sing}.  
\end{itemize}
All in all we get $s(3,2)=a(0)+a(1)=2$.
\end{example}

\begin{example}

Let us compute $s(4)$. Since $n=4$, the only possible ranks are $r=2$ or $r=3$.  Observe that if $r=2$, then $m_1=m_2 = 2$ and $d_1=0$.  Therefore, $\tau = 1$ and since $a(0) = 1$, we get that $s(4,2)=a(0)= 1$.  The corresponding matrix is completely singular and it is the following:
$$
\begin{pmatrix}
0&0&1&0\\
0&0&0&1\\
0&0&0&0\\
0&0&0&0
\end{pmatrix}.$$

On the other hand, if $r=3$, then $m_1 = m_2 = n-r = 1$.  Now, $2r-n+2 = 4$, which means that $\vec{m}=(m_3, m_4)$ where $m_3+m_4\leq 2$ and $m_3\leq 1$. Observe that $\mathcal D = \{(0,0),\, (1,0),\, (1,1)\}$.  Let us study case by case:
\begin{itemize}
\item $(m_3, m_4) = (0,0)$.  Then, $\tau=1$ and $M_{(\tau)}=M_{(1)}$ is a regular matrix of dimension $d_1=2$.  %Recall that $a(2) = 4$.
\item $(m_3, m_4) = (1,0)$.  Now, $\tau=2$ and $M_{(\tau)}=M_{(2)}$ is a regular matrix of dimension $d_2=1$. %Recall that $a(1) = 1$.
\item $(m_3, m_4) = (1,1)$.  In this case, $d_2=0$. % and $a(0) = 1$. 
\end{itemize}
The previous discussion implies that  $s(4,3) = a(2)+a(1)+a(0) = 4 + 1 + 1 = 6$, i.e., there exist 6 new global decomposition of blocks given by the following matrices:
$$\left(\begin{array}{c|cc}
M & \begin{pmatrix}
0&0\\ 0&0
\end{pmatrix}\\
\hline
0 & 0 &1\\
\hline
0 & 0 &0\\
\end{array}\right),\quad 
\begin{pmatrix}
B&0&0&0\\
0&0&1&0\\
0&0&0&1\\
0&0&0&0
\end{pmatrix},\quad 
\begin{pmatrix}
0&1&0&0\\
0&0&1&0\\
0&0&0&1\\
0&0&0&0
\end{pmatrix}.
$$
Here, $M$ is given in Lemma~\ref{lema: block-dim6} and $B\in\mathbb C^*$.

With all this information, we can state that $$s(4) = s(4,2) + s(4,3) = 1+6 = 7.$$
\end{example}

\subsubsection{Unified framework and the general case}\label{sect:general-todos} In this section we estimate the number of different families of 1-abelian complex structures of maximal type that there exist in dimension $2n$.  The results in the previous subsections give rise to the following

\begin{lemma}
    The number of decompositions of $M \in M_{n-1}(\mathbb{C})$ as direct sum of blocks $\Delta_i, \, H_{2i}(\mu)$ and $J_i(0)$ is $$B(n-1)=a(n-1)+s(n-1).$$
\end{lemma}

Finally, using \eqref{AB-B} we get the desired bound:
 
% In order to provide a better bound for $Ab_1(n)$, our idea consists on counting separately the number of each type of 1-abelian complex structure (Hermitian, regular non-Hermitian and singular non-Hermitian) that are of maximal type, and sum all of them together.
\begin{theorem}\label{thm:cota-complex-st}
    The number of families of non-equivalent 1-abelian complex structures $J_M$ of maximal type, $Ab_1(n)$, on NLAs of dimension $2n$ is bounded below by: $$Ab_1(n) \geq \lfloor \frac{n+1}{2} \rfloor+a(n-1)-1+ s(n-1).$$
\end{theorem} 
The summand $\lfloor \frac{n+1}{2} \rfloor$ corresponds to the number of Hermitian complex structures; $a(n-1)-1$ counts the regular non-Hermitian case and, $s(n-1)$ corresponds to the singular non-Hermitian case of maximal type.
%\begin{proof}
   % According to Lemma \ref{hermcount}, for a $2n$-dimensional NLA there are $\lfloor \frac{n+1}{2} \rfloor$ new 1-abelian complex structures $J_M$ with $M \in M_{n-1}(\mathbb{C})$ Hermitian. It remains to sum the new 1-abelian complex structures $J_M$ when $M$ is non-Hermitian. On the one hand, the number of $J_M$ for $M \in M_{n-1}(\mathbb{C})$ regular and non-Hermitian is bigger or equal than the number of decompositions of $M$ as a direct sum of blocks $H_{2k}(\mu)$ and $\lambda\Delta_k$. By section \ref{sect:general-regular}, this number is $a(n-1)$. However, passing to 1-abelian complex structures, the case when $M$ is a multiple of a Hermitian matrix is reduced to the hermitian case, concluding that the number of new $J_M$ for $M$ regular non-Hermitian is $a(n-1)-1$. For the singular case, the idea is to sum new 1-abelian complex structures for each rank smaller than $n-1$. This number is codified by $s(n, r)$ (see Theorem \ref{th4.13}).
%\end{proof}

\begin{example}
Applying the previous theorem and the computations done in the examples above, we conclude that the number of families of non-equivalent 1-abelian complex structures $J_M$ of maximal type on NLAs of dimension $10$ is at least $30$.
\end{example}

\section*{Acknowledgments}
\noindent 
This work has been partially supported by grant 
PID2023-148446NB-I00, funded by\\  MICIU/AEI/10.13039/501100011033, 
and by grant E22-23R ``\'Algebra y Geometr\'ia'' (Gobierno de Arag\'on/ FEDER).

%%%%%%%%%%%%%%%%%%%%%

\end{document}